\documentclass{article}

\usepackage{graphicx}
\usepackage{amsmath, amssymb, amsthm, enumerate}

\usepackage{amsmath,amssymb,amsthm,enumerate,mathrsfs}
\newtheorem{theorem}{Theorem}[section]
\newtheorem{corollary}[theorem]{Corollary}
\newtheorem{lemma}[theorem]{Lemma}

\theoremstyle{definition}

\newtheorem{assumption}{Assumption}[section]

\newtheorem{remark}[theorem]{Remark}
\usepackage{flowchart}
\usepackage{subfig}
\usepackage{algorithm}
\usepackage{algorithmic}
\usepackage{diagbox}

\numberwithin{equation}{section}

\begin{document}
\makeatletter

\begin{center}
\large{\bf Splitting Method for Support Vector Machine in Reproducing Kernel Banach Space with Lower Semi-continuous Loss Function}
\end{center}\vspace{5mm}

\begin{center}
\textsc{Mingyu Mo, Yimin Wei, Qi $\text{\normalfont{Ye}}^{*}$}
\end{center}

\vspace{2mm}

\footnotesize{
\noindent\begin{minipage}{14cm}
{\bf Abstract:}
   In this paper, we use the splitting method to solve support vector machine in reproducing kernel Banach space with lower semi-continuous loss function. We equivalently transfer support vector machines in reproducing kernel Banach space with lower semi-continuous loss function to a finite-dimensional tensor Optimization and propose the splitting method based on alternating direction method of multipliers. By Kurdyka-Lojasiewicz inequality, the iterative sequence obtained by this splitting method is globally convergent to a stationary point if the loss function is lower semi-continuous and subanalytic. Finally, several numerical performances demonstrate the effectiveness.
\end{minipage}
 \\[5mm]



\hbox to14cm{\hrulefill}\par




  \section{Introduction}
   \label{sec:1}
    Support vector machine (SVM) is a successful model in machine learning. The basic idea of SVM is to find a functional in a kernel-based function space that achieves the smallest regularized possible empirical risk and build the SVM by this functional. The SVM is already achieved in some reproducing kernel Hilbert spaces (RKHS) with convex loss function (See \cite{MR2450103}). Recently, paper \cite{Mo2022} discusses the splitting method for the SVM in RKHS with lower semi-continuous loss function. After the success of the SVM in RKHS, people begin to discuss the SVM in reproducing kernel Banach space (RKBS) because RKHS is a special case of RKBS (See \cite{Huang2019,Xu2019}), and paper \cite{Lin2022} proposes a homotopy method for multilinear system induced from RKBS, but the algorithms for the SVM in RKBS are still lack of study. Currently, people are most interested in the infinite-dimensional spaces for applications of machine learning such that the learning algorithms can be chosen from the enough large amounts of suitable solutions. In this paper, for convenient coding and computation, we mainly discuss how to solve the Optimization (\ref{2.1}) induced from the SVM in an infinite-dimensional $\frac{2m}{2m-1}$-norm RKBS with lower semi-continuous loss function by splitting method (See Section \ref{sec:2}). 

    First, we generalize the representer theorem to show that Optimization (\ref{2.1}) has a minimizer in a finite-dimensional closed subset and Optimization (\ref{2.1}) can be equivalently transferred to a finite-dimensional tensor Optimization (\ref{2.5}). From this equivalent Optimization (\ref{2.5}), we discuss the splitting method based on alternating direction method of multipliers (ADMM) for Optimization (\ref{2.1}). By this splitting method, we obtain two subproblems which are computable easily. Also, the convergence of ADMM is already guaranteed well for the convex Optimizations (see \cite{Boyd2011Distributed}) and some special nonconvex Optimizations by the Kurdyka-Lojasiewicz (KL) property (see \cite{Guo2016Convergence,Li2015Global}). To complete this proof of the splitting method, we reexchange the convergence theorems in \cite{Guo2016Convergence,Li2015Global} and verify the convergence of the splitting method for Optimization (\ref{2.1}) if the loss function is lower semi-continuous and subanalytic for the global convergence to a stationary point. At the same time, we give an example of minimizing the sum of two lower semi-continuous and subanalytic functions with nonlinear constraint by ADMM.

    \vspace{-0.2cm}
    \begin{figure}[H]
    \centering
    \begin{tikzpicture}
    \tikzstyle{arrow} = [thick,->,>=stealth]
    \tikzstyle{arrow1} = [dash dot,->,>=stealth]
    \node(1) at (0,0) [draw,rectangle,align=center,minimum width=2cm,minimum height=1cm]{Infinite-dimensional Optimization  \\ (\ref{2.1}) in $\frac{2m}{2m-1}$-norm RKBS};       
    \node(2) at (7.6,0) [draw,rectangle,align=center,minimum width=2cm,minimum height=1cm]{Finite-dimensional Optimizaition \\ (\ref{2.5}) in Euclidean space};        
    \node(3) at (7.6,-1.75)[draw,rectangle,align=center,minimum width=2cm, minimum height=1cm]{Iterative sequence \\ in Euclidean Space}; 
    \node(4) at (3.8,-1.75)[draw,rectangle,align=center,minimum width=2cm, minimum height=1cm]{Iterative sequence \\ in $\frac{2m}{2m-1}$-norm RKBS};
    \node(5) at (0,-1.75)[draw,rectangle,align=center,minimum width=2cm, minimum height=1cm]{The minimizer of \\ Optimization (\ref{2.1})};
    \draw[arrow] (2) -- (3);
    \draw[arrow] (1) -- node[above]{Representer}(2);
    \draw[arrow] (1) -- node[below]{Theorem}(2);
    \draw[arrow] (1) -- (2);
    \draw[arrow] (2) -- node[left]{ADMM}(3);
    \draw[arrow] (3) -- (4);
    \draw[arrow] (4) -- (5);
    \draw[arrow1] (5) -- (1);
    \end{tikzpicture}
    \caption{The Basic Idea of Splitting Method for Optimization (\ref{2.1}).}
    \end{figure}
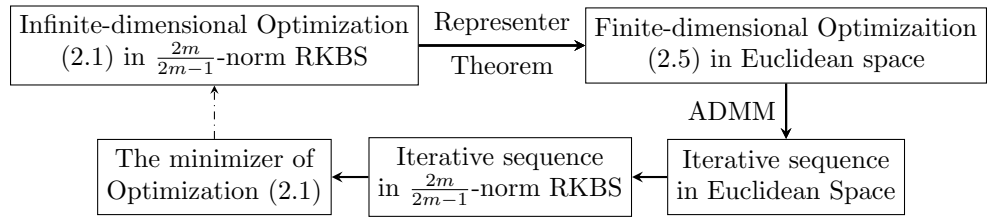
    \vspace{-0.5cm}


   This paper is organized as follows. We introduce some preliminary materials of the SVM in $\frac{2m}{2m-1}$-norm RKBS in Section \ref{sec:2}. Next, we study the splitting method based on ADMM for Optimization (\ref{2.1}) in Section \ref{sec:3}. Moreover, we discuss the global convergence of the iterative algorithm for lower semi-continuous and subanalytic loss function in Section \ref{sec:4}. Finally, we give some numerical examples for synthetic data and real data in Section \ref{sec:5} to show that the SVM in RKBS with lower semi-continuous loss function is better than the SVM in RKHS with convex loss function in some cases.
 
   \section{Support Vector Machines in $\frac{2m}{2m-1}$-norm Reproducing Kernel Banach Space}
   \label{sec:2}
    In this section, we review some preliminaries of the SVM in $\frac{2m}{2m-1}$-norm RKBS. We denote the set of positive integers as $\mathbb{N}$, the set of natural numbers as $\mathbb{N}_{0}$ and the finite set $\{1,2,...,N\}$ as $\mathbb{N}_{N}$ respectively. Also, $\mathbb{N}^{d}$ is the tensor product of positive integers and $\mathbb{R}^{d}$ is the $d$-dimensional Euclidean space. For the rest of this paper, without specification, every vector is supposed to be a column vector. 

    \subsection{$\frac{2m}{2m-1}$-norm Reproducing Kernel Banach Space}
    \label{Sec:2.1}
    For convenient coding and computation, in this subsection, we review the basic concepts of infinite-dimensional $\frac{2m}{2m-1}$-norm separable RKBS (See \cite[Section 3.2]{Xu2019}), $m\in \mathbb{N}$. For a given nonempty subset $X\subseteq \mathbb{R}^{d}$ and a linearly independent basis $\phi_{n}:X\to \mathbb{R},\ n=1,2,...,$ such that 
    $$
    \sum_{n\in \mathbb{N}}|\phi_{n}(\boldsymbol{x})|<\infty\ \text{for all}\ \boldsymbol{x}\in X, 
    $$
we introduce the following $\frac{2m}{2m-1}$-norm Banach space of continuous functions over $X$, that is,  
    $$
    \mathcal{B}^{\frac{2m}{2m-1}}_{K}(X):=\left\{f:=\sum\limits_{n\in \mathbb{N}} a_{n}\phi_{n}:\ a_{n}\in \mathbb{R},\ \sum_{n\in \mathbb{N}} |a_{n}|^{\frac{2m}{2m-1}}<\infty \right\},
    $$
equipped with the norm $\|f\|_{\mathcal{B}^{\frac{2m}{2m-1}}_{K}(X)}:=\left(\sum\limits_{n\in \mathbb{N}} |a_{n}|^{\frac{2m}{2m-1}}\right)^{\frac{2m-1}{2m}}$ and the well-defined kernel $K:X\times X\to \mathbb{R}$, and
    $$
       K(\boldsymbol{x},\boldsymbol{x}'):=\sum\limits_{n\in \mathbb{N}} \phi_{n}(\boldsymbol{x})\phi_{n}(\boldsymbol{x}'),\ \forall \boldsymbol{x},\boldsymbol{x}'\in X.
    $$
Let $l^{\frac{2m}{2m-1}}$ and $l^{2m}$ be the set of all countable sequences of real scalars with the standard $\frac{2m}{2m-1}$-norm and $2m$-norm, respectively. This construction of $\mathcal{B}^{\frac{2m}{2m-1}}_{K}(X)$ ensures that $\mathcal{B}^{\frac{2m}{2m-1}}_{K}(X)$ is isometrically isomorphic to $l^{\frac{2m}{2m-1}}$, that is,
$\mathcal{B}^{\frac{2m}{2m-1}}_{K}(X)\cong l^{\frac{2m}{2m-1}}$. Let $(l^{\frac{2m}{2m-1}})^{'}$ and $(\mathcal{B}^{\frac{2m}{2m-1}}_{K}(X))^{'}$ be the dual space of $l^{\frac{2m}{2m-1}}$ and $\mathcal{B}^{\frac{2m}{2m-1}}_{K}(X)$, respectively. Since $(l^{\frac{2m}{2m-1}})^{'}\cong l^{2m}$, we have that
    $$
    (\mathcal{B}^{\frac{2m}{2m-1}}_{K}(X))^{'}\cong \mathcal{B}^{2m}_{K}(X):=\left\{g:=\sum\limits_{n\in \mathbb{N}} b_{n}\phi_{n}:\ b_{n}\in \mathbb{R},\ \sum_{n\in \mathbb{N}} |b_{n}|^{2m}<\infty \right\},
    $$
and $\mathcal{B}^{2m}_{K}(X)$ is equipped with the norm $\|g\|_{\mathcal{B}^{2m}_{K}(X)}:=\left(\sum\limits_{n\in \mathbb{N}} |b_{n}|^{2m}\right)^{\frac{1}{2m}}$ and the same kernel $K$. Moreover, we find that the dual bilinear product of $\mathcal{B}^{\frac{2m}{2m-1}}_{K}(X)$ is consistent with the dual bilinear product of $l^{\frac{2m}{2m-1}}$, that is, 
    $$
    \langle f,g \rangle_{\mathcal{B}^{\frac{2m}{2m-1}}_{K}(X)}=\langle \{a_{n}\},\{b_{n}\} \rangle_{l^{\frac{2m}{2m-1}}}   =\sum\limits_{n\in \mathbb{N}} a_{n}b_{n},
    $$ 
for all $f\in \mathcal{B}^{\frac{2m}{2m-1}}_{K}(X)$ and $g\in \mathcal{B}^{2m}_{K}(X)$. Next we check the two-sided reproducing property of $\mathcal{B}^{\frac{2m}{2m-1}}_{K}(X)$, that is, for all $\boldsymbol{x},\boldsymbol{x}'\in X$, $f\in \mathcal{B}^{\frac{2m}{2m-1}}_{K}(X)$, $g\in \mathcal{B}^{2m}_{K}(X)$,
    $$
       K(\boldsymbol{x},\cdot)=\sum_{n\in \mathbb{N}} \phi_{n}(\boldsymbol{x})\phi_{n} \in \mathcal{B}^{2m}_{K}(X),\ \ \langle f, K(\boldsymbol{x},\cdot)\rangle_{\mathcal{B}^{\frac{2m}{2m-1}}_{K}(X)}=\sum_{n\in \mathbb{N}} a_{n}\phi_{n}(\boldsymbol{x})=f(\boldsymbol{x}),
    $$          
and
    $$
       K(\cdot,\boldsymbol{x}')=\sum_{n\in \mathbb{N}} \phi_{n}(\boldsymbol{x}')\phi_{n}\in \mathcal{B}^{\frac{2m}{2m-1}}_{K}(X),\ \ \langle K(\cdot,\boldsymbol{x}'),g\rangle_{\mathcal{B}^{\frac{2m}{2m-1}}_{K}(X)}=\sum_{n\in \mathbb{N}} b_{n}\phi_{n}(\boldsymbol{x}')=g(\boldsymbol{x}').
    $$
Specially, for any $m\in \mathbb{N}$, $\mathcal{B}^{\frac{2m}{2m-1}}_{K}(X)$ has the same reproducing kernel $K$. In particular, when $m=1$, $\frac{2m}{2m-1}=2$ and $\mathcal{B}_{K}^{2}(X)$ is consistent with the separable RKHS.

    \subsection{Support Vector Machines in $\frac{2m}{2m-1}$-norm Reproducing Kernel Banach Space}
    \label{Sec:2.2}
    In this subsection, we introduce some preliminaries of the SVM in $\mathcal{B}^{\frac{2m}{2m-1}}_{K}(X)$. Let $X$ and $Y$ be the sample space in $\mathbb{R}^{d}$ and the label space in $\mathbb{R}$, respectively. We have the training data 
    $$
      D:=\{(\boldsymbol{x}_{i},y_{i}):i=1,2,...,N\}\subseteq X\times Y, 
    $$
composed of input data $\boldsymbol{x}_{1},\boldsymbol{x}_{2},...,\boldsymbol{x}_{N}\in X$ and output data $y_{1},y_{2},...,y_{N}\in Y$. To avoid overfitting, we use the training data $D$ to learn a function $f_{D}:X\to \mathbb{R}$ that achieves the smallest regularized possible empirical risk in a given kernel-based function space over $X$ and build the corresponding SVM $\mathcal{R} f_{D}:X\to Y$ to predict the label of $\boldsymbol{x}\neq \boldsymbol{x}_{i},\ i=1,2,...,N$. According to the task requirements, we use $f_{D}$ to construct different SVMs. For example, if $Y=\mathbb{R}$, then we build $\mathcal{R} f_{D}=f_{D}$ for regression. If $Y=\{+1,-1\}$, then we can build the following SVM
    $$
        \mathcal{R} f_{D}(\boldsymbol{x})=\begin{cases} +1, & f_{D}(\boldsymbol{x}) \geq0,\\
                                 -1, & f_{D}(\boldsymbol{x})<0.
                   \end{cases}   
    $$
for binary classification (See \cite[Section 8 and 9]{MR2450103}). The classical SVM is discussed in a $d$-dimensional RKHS consisting all linear functionals on $X$ with linear kernel (see \cite{Cortes1995} and \cite[Section 1.3]{MR2450103}). For more flexible kernels such as those of Gaussian kernels, which belong to the most important kernels in practice, the corresponding RKHS is infinite-dimensional such that the learning algorithms can be chosen from the enough large amounts of suitable solutions (see in \cite[Chapter 5]{MR2450103}). The SVM in RKHS has two main areas of application: risk minimization in machine learning and data mining. The typical examples of the SVM in RKHS are the price assessment of a house based on certain characteristics for regression and the automatic recognition of hand-written digits for classification. After the success of the SVM in RKHS, people begin to discuss the SVM in RKBS because RKHS is a special case of RKBS. 

    Since $\mathcal{B}_{K}^{2}(X)$ is a separable RKHS with reproducing kernel $K$, the SVM in $\mathcal{B}_{K}^{2}(X)$ is consistent with the SVM in separable RKHS. Recently, \cite[Section 5.4]{Xu2019} generalizes the SVM in $\mathcal{B}_{K}^{2}(X)$ to $\mathcal{B}^{\frac{2m}{2m-1}}_{K}(X)$. For a given $\mathcal{B}^{\frac{2m}{2m-1}}_{K}(X)$, the training data $D$ and $\lambda>0$, we will find a function in $\mathcal{B}^{\frac{2m}{2m-1}}_{K}(X)$ that achieves the smallest regularized possible empirical risk, that is,
    \begin{equation}\label{2.1}
      \mathop{\mathrm{inf}}_{f\in \mathcal{B}^{\frac{2m}{2m-1}}_{K}(X)}\ \frac{1}{N}\sum_{i\in \mathbb{N}_{N}} L(\boldsymbol{x}_{i},y_{i},f(\boldsymbol{x}_{i}))+\lambda \|f\|_{\mathcal{B}^{\frac{2m}{2m-1}}_{K}(X)}^{\frac{2m}{2m-1}},
    \end{equation}
where $L:X\times Y\times \mathbb{R}\rightarrow [0,\infty)$ is a given loss function and $\lambda \|f\|_{\mathcal{B}^{\frac{2m}{2m-1}}_{K}(X)}^{\frac{2m}{2m-1}}$ is the regularization term used to penalize $f$ with the large RKBS norm. In the following, we will interpret $L(\boldsymbol{x},y,f(\boldsymbol{x}))$ as the loss of predicting $y$ by $f(\boldsymbol{x})$ if $\boldsymbol{x}$ is observed, that is, the smaller the value $L(\boldsymbol{x}, y,f(\boldsymbol{x}))$ is, the better $f(\boldsymbol{x})$ predicts $y$ in the sense of $L$ (See Section \ref{sec:5} for more detail). It is clear that Optimization (\ref{2.1}) is infinite-dimensional and nonnegative. We denote the minimizer of Optimization (\ref{2.1}) as $f_{D}^{\frac{2m}{2m-1}}$, then we build the SVM $\mathcal{R} f_{D}^{\frac{2m}{2m-1}}:X\to Y$ according to the task requirement. Moreover, SVMs constructed by different minimizers of Optimization (\ref{2.1}) have no difference in performance. Hence, we just need to find a minimizer of Optimization (\ref{2.1}). Next, we focus on the minimizer of Optimization (\ref{2.1}).

    \subsection{Duality Mapping and Tensor}
    \label{sec:2.3}
    To discuss the minimizer of Optimization (\ref{2.1}), we review some concepts of duality mapping and tensor in $\mathcal{B}^{\frac{2m}{2m-1}}_{K}(X)$ (See \cite[Section 2.1]{Unser2020} and \cite[Section 3]{Ye2022}). Recall that a mapping (not necessarily linear) $\Psi_{K}^{\frac{2m}{2m-1}}:\mathcal{B}^{\frac{2m}{2m-1}}_{K}(X)\to \mathcal{B}^{2m}_{K}(X)$ is called duality mapping if 
    $$
      \|\Psi_{K}^{\frac{2m}{2m-1}}(f)\|_{\mathcal{B}^{2m}_{K}(X)}=\|f\|_{\mathcal{B}^{\frac{2m}{2m-1}}_{K}(X)}^{\frac{1}{2m-1}}\ \text{and}\ \langle f,\Psi_{K}^{\frac{2m}{2m-1}}(f) \rangle_{\mathcal{B}^{\frac{2m}{2m-1}}_{K}(X)}=\|f\|_{\mathcal{B}^{\frac{2m}{2m-1}}_{K}(X)}^{\frac{2m}{2m-1}}.
    $$
Since $l^{\frac{2m}{2m-1}}$ is reflexive, strictly convex and smooth, $\mathcal{B}^{\frac{2m}{2m-1}}_{K}(X)$ is also reflexive, strictly convex and smooth. Hence, $\Psi_{K}^{\frac{2m}{2m-1}}$ is a homeomorphism, and
    \begin{align*}
        \Psi_{K}^{\frac{2m}{2m-1}}(f)&=\sum_{n\in \mathbb{N}} (a_{n})^{\frac{1}{2m-1}}\phi_{n},\ \forall f\in \mathcal{B}^{\frac{2m}{2m-1}}_{K}(X), \\
        (\Psi_{K}^{\frac{2m}{2m-1}})^{-1}(g)&=\sum\limits_{n\in \mathbb{N}}(b_{n})^{2m-1}\phi_{n},\ \forall g\in \mathcal{B}^{2m}_{K}(X).
    \end{align*}
If $m=1$, then $\Psi_{K}^{2}$ is an identity mapping and thus linear from $\mathcal{B}_{K}^{2}(X)$ onto $\mathcal{B}_{K}^{2}(X)$. However, if $m>1$, then $\Psi_{K}^{\frac{2m}{2m-1}}$ is nonlinear but continuous. Furthermore, we obtain the following formula of Fr\'echet derivative (see the proof of \cite[Theorem 5.2]{Xu2019})
    \begin{equation}\label{2.2}
       \nabla(\|\cdot\|_{\mathcal{B}^{\frac{2m}{2m-1}}_{K}(X)})(f)=  \frac{\Psi_{K}^{\frac{2m}{2m-1}}(f)} {\|f\|_{\mathcal{B}^{\frac{2m}{2m-1}}_{K}(X)}^{\frac{1}{2m-1}}}   \in \mathcal{B}^{2m}_{K}(X),\ \ \forall f\neq 0,
    \end{equation}
where $\nabla$ denotes the Fr\'echet derivative. In particular, for a differentiable real function, $\nabla$ represents the gradient and $\nabla^{2}$ represents the Hessian matrix, respectively.

    Next we review some basic theory of tensor in $\mathcal{B}^{\frac{2m}{2m-1}}_{K}(X)$. For convenience of readers, the notations and operations of tensors are defined as in the book \cite{Qi2018}. For the training data $D$, let 
    $$
       \boldsymbol{\boldsymbol{\Phi}}_{n}:=\left(\phi_{n}(\boldsymbol{x}_{1}),\phi_{n}(\boldsymbol{x}_{2}),..,\phi_{n}(\boldsymbol{x}_{N})\right)^{T}\in \mathbb{R}^{N},\ n\in \mathbb{N}. 
    $$
Since $2m$ is even, we define the following $(2m)$-th order $N$-dimensional real tensor in $\mathcal{B}^{\frac{2m}{2m-1}}_{K}(X)$
    $$
       \mathcal{A}_{K}^{2m}:=\left(\sum\limits_{n\in \mathbb{N}} \phi_{n}(\boldsymbol{x}_{i_{1}})\phi_{n}(\boldsymbol{x}_{i_{2}})...\phi_{n}(\boldsymbol{x}_{i_{2m}})\right)_{i_{1},i_{2},...,i_{2m}=1}^{N,N,...,N}=\sum\limits_{n\in \mathbb{N}}(\boldsymbol{\Phi}_{n})^{\otimes 2m},
    $$
where $\otimes$ denotes the tensor outer product. Specially, if $m=1$, then $\mathcal{A}_{K}^{2}$ is a matrix in $\mathcal{B}_{K}^{2}(X)$. It is easy to check that $\mathcal{A}_{K}^{2m}$ is symmetric, that is, all entries of $\mathcal{A}_{K}^{2m}$ are invariant under any permutation of the indices. For any $\boldsymbol{c}\in \mathbb{R}^{N}$, we denote
    \begin{align*}
        \mathcal{A}_{K}^{2m}\boldsymbol{c}^{2m}:=&\left(\sum\limits_{n\in \mathbb{N}}(\boldsymbol{\Phi}_{n})^{\otimes 2m}\right)\cdot \boldsymbol{c}^{\otimes 2m}=\sum_{n\in \mathbb{N}} \left(\boldsymbol{\Phi}_{n}^{T}\boldsymbol{c}\right)^{2m}\geq 0, \\
        \mathcal{A}_{K}^{2m}\boldsymbol{c}^{2m-1}:=&\left(\sum\limits_{n\in \mathbb{N}}(\boldsymbol{\Phi}_{n})^{\otimes 2m}\right)\cdot \boldsymbol{c}^{\otimes 2m-1}=\sum_{n\in \mathbb{N}} \left(\boldsymbol{\Phi}_{n}^{T}\boldsymbol{c}\right)^{2m-1}\boldsymbol{\Phi}_{n}\in \mathbb{R}^{N}, \\
        \mathcal{A}_{K}^{2m}\boldsymbol{c}^{2m-2}:=&\left(\sum\limits_{n\in \mathbb{N}}(\boldsymbol{\Phi}_{n})^{\otimes 2m}\right)\cdot \boldsymbol{c}^{\otimes 2m-2}=\sum_{n\in \mathbb{N}} \left(\boldsymbol{\Phi}_{n}^{T}\boldsymbol{c}\right)^{2m-2}\boldsymbol{\Phi}_{n}\boldsymbol{\Phi}_{n}^{T}\in \mathbb{R}^{N\times N}.
    \end{align*}
Specially, we adopt the convention that $\mathcal{A}_{K}^{2}\boldsymbol{c}^{0}=\mathcal{A}_{K}^{2}$ for any $\boldsymbol{c}\neq \boldsymbol{0}$. It is clear that
    \begin{equation}\label{2.3}
      \mathcal{A}_{K}^{2m}\boldsymbol{c}^{2m}=(\mathcal{A}_{K}^{2m}\boldsymbol{c}^{2m-1})^{T} \boldsymbol{c},\ \ \mathcal{A}_{K}^{2m}\boldsymbol{c}^{2m-1}=\mathcal{A}_{K}^{2m}\boldsymbol{c}^{2m-2}\cdot \boldsymbol{c}.
    \end{equation}
Also, by derivative rule, we have that for any $\boldsymbol{c}\in \mathbb{R}^{N}$,
    \begin{equation}\label{2.4}
       \nabla(\mathcal{A}_{K}^{2m}(\cdot)^{2m})(\boldsymbol{c})=2m\mathcal{A}_{K}^{2m}\boldsymbol{c}^{2m-1},\ \ \nabla^{2}(\mathcal{A}_{K}^{2m}(\cdot)^{2m})(\boldsymbol{c})=2m(2m-1)\mathcal{A}_{K}^{2m}\boldsymbol{c}^{2m-2}.
    \end{equation}
By definition, it is easy to check that $\mathcal{A}_{K}^{2m}\boldsymbol{c}^{2m-2}$ is symmetric. For any $\boldsymbol{d}\in \mathbb{R}^{N}$, we have that
    \begin{equation*}
       \boldsymbol{d}^{T}\left(\mathcal{A}_{K}^{2m}\boldsymbol{c}^{2m-2}\right)\boldsymbol{d}=\sum_{n\in \mathbb{N}}\left(\boldsymbol{\Phi}_{n}^{T}\boldsymbol{c}\right)^{2m-2}\left(\boldsymbol{\Phi}_{n}^{T}\boldsymbol{d}\right)^{2}\geq 0.
    \end{equation*}
Thus, $\mathcal{A}_{K}^{2m}\boldsymbol{c}^{2m-2}$ is a positive definite matrix.

    \subsection{Representer Theorem}
    The loss functions and RKBSs are the core research content, because it determines the minimizer of Optimization (\ref{2.1}). For machine learning, the most common loss function is convex loss function, that is, $L(\boldsymbol{x}, y,\cdot)$ is a convex function for all $\boldsymbol{x}\in X$ and $y\in Y$. Besides convexity, we can define other loss functions in a similar way, such as continuity, smoothness, lower semi-continuity, etc. Specially, if $L$ is a convex loss function, then $L$ is a lower semi-continuous loss function. But the lower semi-continuity can not imply the convexity.

    If $L$ is a convex loss function, then \cite[Theorem 5.2]{Xu2019} assures that Optimization (\ref{2.1}) in $\mathcal{B}^{\frac{2m}{2m-1}}_{K}(X)$ with convex loss function has a unique minimizer contained in a known finite-dimensional space spanned by the reproducing kernel $K$ and the training data $D$, even if the space $\mathcal{B}^{\frac{2m}{2m-1}}_{K}(X)$ itself is substantially larger. The convex loss function is viewed as highly preferable in many publications because of their computational advantages (unique minimizer, ease-of-use, ability to be efficiently optimized by convex optimization tools, etc.). 

    Recently, some nonconvex and lower semi-continuous loss functions, such as ramp loss and truncated pinball loss are proposed and used in SVM (see \cite{Brooks2011, Feng2016, Liu2016, Shen2017}). From some preliminary numerical experiments in \cite[Chapter 6]{Xu2019}, it shows that the SVM in $\mathcal{B}_{K}^{\frac{4}{3}}(X)$ may perform better than the SVM in $\mathcal{B}_{K}^{2}(X)$ with Hinge loss and Gaussian kernel in some cases. Furthermore, the SVM in $\mathcal{B}_{K}^{\frac{6}{5}}(X)$ may perform better than the SVM in $\mathcal{B}_{K}^{\frac{4}{3}}(X)$ and $\mathcal{B}_{K}^{2}(X)$ with a nonconvex linear piecewise loss function and Gaussian kernel in some cases (See Section \ref{sec:5}). But an effective algorithm for the SVM in $\mathcal{B}^{\frac{2m}{2m-1}}_{K}(X)$ with lower semi-continuous loss function is still lack of study. Thus, in the rest of this paper, we mainly study how to solve Optimization (\ref{2.1}) in $\mathcal{B}^{\frac{2m}{2m-1}}_{K}(X)$ with lower semi-continuous loss function. First, we establish the representer theorem by the same techniques as those used in \cite{Huang2019}. We present it below.

    \begin{lemma}\label{Lemma:2.1}
       If $L$ is a lower semi-continuous loss function, then Optimization (\text{\normalfont{\ref{2.1}}}) has a minimizer $f_{D}^{\frac{2m}{2m-1}}$ such that
    $$
      f_{D}^{\frac{2m}{2m-1}}\in \Delta_{K}^{\frac{2m}{2m-1}}(X):=(\Psi_{K}^{\frac{2m}{2m-1}})^{-1}(\mathrm{span}\{K(\boldsymbol{x}_{1},\cdot),...,K(\boldsymbol{x}_{N},\cdot)\}).
    $$
    \end{lemma}

    \begin{proof}
    From (\ref{Sec:2.1}), $\mathcal{B}^{\frac{2m}{2m-1}}_{K}(X)$ is a two-sided RKBS and thus a right-sided RKBS (See \cite[Definition 2.1]{Xu2019}). By \cite[Definition 2.2.27]{Dales} and reflexity of $\mathcal{B}^{\frac{2m}{2m-1}}_{K}(X)$, the predual space of $\mathcal{B}^{\frac{2m}{2m-1}}_{K}(X)$ is $\mathcal{B}^{2m}_{K}(X)$. Moreover, the right-sided reproducing property of $\mathcal{B}^{\frac{2m}{2m-1}}_{K}(X)$ shows that $\{K(\boldsymbol{x},\cdot):\boldsymbol{x}\in X\}\subseteq \mathcal{B}^{2m}_{K}(X)$. In conclusion, we use $\lambda\|\cdot\|_{\mathcal{B}^{\frac{2m}{2m-1}}_{K}(X)}^{\frac{2m}{2m-1}}$ instead of $\lambda\|\cdot\|_{\mathcal{B}^{\frac{2m}{2m-1}}_{K}(X)}^{2}$ to reproof \cite[Theorem 3.1]{Huang2019} and use similar line of arguments of \cite[Corollary 4.1]{Huang2019} to show that Optimization (\ref{2.1}) has a minimizer $f_{D}^{\frac{2m}{2m-1}}$. Moreover, if $f_{D}^{\frac{2m}{2m-1}}\neq 0$, then (\ref{2.2}) guarantees that
    $$
    \frac{\Psi_{K}^{\frac{2m}{2m-1}}(f_{D}^{\frac{2m}{2m-1}})} {\|f_{D}^{\frac{2m}{2m-1}}\|_{\mathcal{B}^{\frac{2m}{2m-1}}_{K}(X)}^{\frac{1}{2m-1}}} =\nabla(\|\cdot\|_{\mathcal{B}^{\frac{2m}{2m-1}}_{K}(X)})(f_{D}^{\frac{2m}{2m-1}})\in \mathrm{span}\{K(\boldsymbol{x}_{1},\cdot),...,K(\boldsymbol{x}_{N},\cdot)\},
    $$  
which ensures that
    $$
      f_{D}^{\frac{2m}{2m-1}}\in \Delta_{K}^{\frac{2m}{2m-1}}(X)=(\Psi_{K}^{\frac{2m}{2m-1}})^{-1}(\mathrm{span}\{K(\boldsymbol{x}_{1},\cdot),...,K(\boldsymbol{x}_{N},\cdot)\}).
    $$
If $f_{D}^{\frac{2m}{2m-1}}=0$, then the proof is straight-forward. This completes the proof.
    \end{proof}

    Since $\Psi_{K}^{\frac{2m}{2m-1}}$ is a homeomorphism and $\mathrm{span}\{K(\boldsymbol{x}_{1},\cdot),...,K(\boldsymbol{x}_{N},\cdot)\}$ is a finite-dimensional subspace and thus a closed subset in $\mathcal{B}^{2m}_{K}(X)$, $\Delta_{K}^{\frac{2m}{2m-1}}(X)$ is a closed subset in $\mathcal{B}^{\frac{2m}{2m-1}}_{K}(X)$. If $m=1$, then $\Psi_{K}^{2}$ is linear and $\Delta_{K}^{2}(X)$ is a subspace. If $m>1$, then $\Psi_{K}^{\frac{2m}{2m-1}}$ is nonlinear and $\Delta_{K}^{\frac{2m}{2m-1}}(X)$ is not a subspace.

    \begin{remark}
      If $L$ is a lower semi-continuous and nonconvex loss function, then Optimization (\ref{2.1}) may have more than one minimizer and at least one of them is in $\Delta_{K}^{\frac{2m}{2m-1}}(X)$. Hence, we focus on finding the minimizer in $\Delta_{K}^{\frac{2m}{2m-1}}(X)$.
   \end{remark}

    \subsection{Tensor Optimization}
    In this subsection, we introduce how to equivalently transfer Optimization (\ref{2.1}) to a finite-dimensional tensor Optimization. First, Lemma \ref{Lemma:2.1} shows that Optimization (\ref{2.1}) in $\mathcal{B}^{\frac{2m}{2m-1}}_{K}(X)$ can be equivalently transferred to Optimization (\ref{2.1}) in $\Delta_{K}^{\frac{2m}{2m-1}}(X)$. For any $f\in \Delta_{K}^{\frac{2m}{2m-1}}(X)$, there exists a vector $\boldsymbol{c}=(c_{1},c_{2}...,c_{N})^{T}\in \mathbb{R}^{N}$ such that
    $$
       \Psi_{K}^{\frac{2m}{2m-1}}(f)=\sum_{i\in \mathbb{N}_{N}}c_{i}K(\boldsymbol{x}_{i},\cdot)=\sum_{n\in \mathbb{N}}\left(\boldsymbol{\Phi}_{n}^{T}\boldsymbol{c}\right)\phi_{n}.
    $$
Thus, $f$ has the representation
    $$
       f=(\Psi_{K}^{\frac{2m}{2m-1}})^{-1}\left(\sum_{n\in \mathbb{N}}\left(\boldsymbol{\Phi}_{n}^{T}\boldsymbol{c}\right)\phi_{n}\right)=\sum_{n\in \mathbb{N}} \left(\boldsymbol{\Phi}_{n}^{T}\boldsymbol{c}\right)^{2m-1}\phi_{n},
    $$
which ensures that
    $$
       f(\boldsymbol{x}_{i})=\sum_{n\in \mathbb{N}} \left(\boldsymbol{\Phi}_{n}^{T}\boldsymbol{c}\right)^{2m-1}\phi_{n}(\boldsymbol{x}_{i})=(\mathcal{A}_{K}^{2m}\boldsymbol{c}^{2m-1})_{i},\ i=1,2,...,N.
    $$
On the other hand, by definition of $\|\cdot\|_{\mathcal{B}^{\frac{2m}{2m-1}}_{K}(X)}$, for any $f\in \Delta_{K}^{\frac{2m}{2m-1}}(X)$, we compute
    $$
       \|f\|_{\mathcal{B}^{\frac{2m}{2m-1}}_{K}(X)}^{\frac{2m}{2m-1}}=\sum_{n\in \mathbb{N}} \left| \left(\boldsymbol{\Phi}_{n}^{T}\boldsymbol{c}\right)^{2m-1}\right|^{\frac{2m}{2m-1}}=\sum_{n\in \mathbb{N}} \left(\boldsymbol{\Phi}_{n}^{T}\boldsymbol{c}\right)^{2m}=\mathcal{A}_{K}^{2m}\boldsymbol{c}^{2m}.
    $$
Thus, Optimization (\ref{2.1}) in $\Delta_{K}^{\frac{2m}{2m-1}}(X)$ can be equivalently transferred to a tensor Optimization in $\mathbb{R}^{N}$
    \begin{equation}\label{2.5}
       \min_{\boldsymbol{c}\in \mathbb{R}^{N}}\ \frac{1}{N}\sum_{i\in \mathbb{N}_{N}} L(\boldsymbol{x}_{i},y_{i},(\mathcal{A}_{K}^{2m}\boldsymbol{c}^{2m-1})_{i})+\lambda \mathcal{A}_{K}^{2m}\boldsymbol{c}^{2m}.
    \end{equation}
Since $f_{D}^{\frac{2m}{2m-1}}\in \Delta_{K}^{\frac{2m}{2m-1}}(X)$ is a minimizer of Optimization (\ref{2.1}) in $\mathcal{B}^{\frac{2m}{2m-1}}_{K}(X)$, there exists $\boldsymbol{c}_{D}^{\frac{2m}{2m-1}}\in \mathbb{R}^{N}$ such that
    $$
       f_{D}^{\frac{2m}{2m-1}}=\sum_{n\in \mathbb{N}} \left(\boldsymbol{\Phi}_{n}^{T}\boldsymbol{c}_{D}^{\frac{2m}{2m-1}}\right)^{2m-1}\phi_{n}
    $$
and $\boldsymbol{c}_{D}^{\frac{2m}{2m-1}}$ is a minimizer of Optimization (\ref{2.5}). This ensures that we employ the finite suitable parameters to reconstruct the SVM in $\mathcal{B}^{\frac{2m}{2m-1}}_{K}(X)$. 

    \section{Splitting Method for Support Vector Machines in $\frac{2m}{2m-1}$-norm Reproducing Kernel Banach Space}
    \label{sec:3}
    In this section, we consider to find an algorithm based on Optimization (\ref{2.5}) to compute Optimization (\ref{2.1}) easily. At present, we mainly use subgradient method, Lagrange multipliers method and sequential minimal optimization (SMO) to design algorithms for SVM. These classical numerical algorithms are suitable for solving convex and smooth programs. Since $\boldsymbol{c}\mapsto (\mathcal{A}_{K}^{2m}\boldsymbol{c}^{2m-1})_{i}$ is continuous, the lower semi-continuity of $L(\boldsymbol{x}_{i},y_{i},\cdot)$ guarantees the lower semi-continuity of $\boldsymbol{c}\mapsto L(\boldsymbol{x}_{i},y_{i},(\mathcal{A}_{K}^{2m}\boldsymbol{c}^{2m-1})_{i})$ for all $i=1,2,...,N$ which ensures Optimization (\ref{2.5}) is lower semi-continuous which maybe nonsmooth or nonconvex. Many classical numerical algorithms are not suitable for Optimization (\ref{2.5}).

    The ADMM algorithm, as one of splitting techniques, can even be used to minimize nonsmooth or nonconvex functions. Also, \cite{Wang2021} discuss how to use ADMM for the traditional SVM with 0-1 loss function. For general lower semi-continuous loss functions and kernels, we observe that the subproblems in ADMM for Optimization (\ref{2.5}) can be splitted into some Optimizations which are then easier to handle. Hence, we will study how to solve Optimization (\ref{2.1}) by the splitting method based on ADMM. 

    \begin{remark}
    In the rest of this paper, we assume that the training data $D$, loss function $L$, RKBS $\mathcal{B}^{\frac{2m}{2m-1}}_{K}(X)$ and $\lambda>0$ in Optimization (\ref{2.1}) and (\ref{2.5}) are given, which means that $N\in \mathbb{N}$, $m\in \mathbb{N}$ and the tensor $\mathcal{A}_{K}^{2m}$ are confirmed before algorithm discussion. 
    \end{remark}

    For notational convenience, let 
    $$
       F(\boldsymbol{\alpha}):=\frac{1}{N}\sum_{i\in \mathbb{N}_{N}} L(\boldsymbol{x}_{i},y_{i},\alpha_{i}), \ \  G(\boldsymbol{c}):=\lambda \mathcal{A}_{K}^{2m}\boldsymbol{c}^{2m}.
    $$
To describe the algorithm, we first reformulate Optimization (\ref{2.5}) as
    \begin{equation*}
    \begin{array}{cl}
    \displaystyle\min_{\boldsymbol a, \boldsymbol c}\enskip & \displaystyle  F(\boldsymbol{\alpha})+G(\boldsymbol{c}),  \\
    \text{s.t.}\enskip & \boldsymbol{\alpha}=\mathcal{A}_{K}^{2m}\boldsymbol{c}^{2m-1}.
    \end{array}
    \end{equation*}
Recall that the augmented Lagrangian function for the above problem is defined as:
    \begin{equation*}
    \mathcal{L}_{\beta}(\boldsymbol{\alpha},\boldsymbol{c},\boldsymbol{\gamma}):=F(\boldsymbol{\alpha})+G(\boldsymbol{c})+\boldsymbol{\gamma}^{T}(\boldsymbol{\alpha}-\mathcal{A}_{K}^{2m}\boldsymbol{c}^{2m-1})+\frac{\beta}{2}\|\boldsymbol{\alpha}-\mathcal{A}_{K}^{2m}\boldsymbol{c}^{2m-1}\|^{2},
    \end{equation*}
where the Lagrangian multiplier $\beta>0$ and $\|\cdot\|$ denotes the $2$-norm in Euclidean space. The splitting method based on ADMM is then presented as follows. Suppose that the algorithm is initialized at $(\boldsymbol{\alpha}^{0},\boldsymbol{c}^{0}, \boldsymbol{\gamma}^{0})$, its iterative scheme is
      \begin{align}
         \boldsymbol{\alpha}_{k+1}&\in \mathop{\mathrm{argmin}}_{\boldsymbol{\alpha}\in \mathbb{R}^{N}}\ \mathcal{L}_{\beta}(\boldsymbol{\alpha},\boldsymbol{c}_{k},\boldsymbol{\gamma}_{k}),\tag{S-1}  \label{S-1} \\
         \boldsymbol{c}_{k+1}&\in \mathop{\mathrm{argmin}}_{\boldsymbol{c}\in \mathbb{R}^{N}}\ \mathcal{L}_{\beta}(\boldsymbol{\alpha}_{k+1},\boldsymbol{c},\boldsymbol{\gamma}_{k}), \tag{S-2} \label{S-2} \\
         \boldsymbol{\gamma}_{k+1}&:=\boldsymbol{\gamma}_{k}+\beta (\boldsymbol{\alpha}_{k+1}-\mathcal{A}_{K}^{2m}(\boldsymbol{c}_{k+1})^{2m-1}), \label{S-3} \tag{S-3} \\
         s_{k+1}&:=\sum_{n\in \mathbb{N}} \left(\boldsymbol{\Phi}_{n}^{T}\boldsymbol{c}_{k+1}\right)^{2m-1}\phi_{n}, \label{S-4} \tag{S-4}
       \end{align}
where $k\in \mathbb{N}_{0}$. Since (\ref{S-1}) only depends on $\boldsymbol{\alpha}$ and (\ref{S-2}) only depends on $\boldsymbol{c}$, combining the linear and quadratic terms, we equivalently transfer (\ref{S-1}) and (\ref{S-2}) to
      \begin{align}
         \boldsymbol{\alpha}_{k+1}&\in \mathop{\mathrm{argmin}}_{\boldsymbol{\alpha}\in \mathbb{R}^{N}}\ F(\boldsymbol{\alpha})+\frac{\beta}{2}\|\boldsymbol{\alpha}-\mathcal{A}_{K}^{2m}(\boldsymbol{c}_{k})^{2m-1}+\frac{1}{\beta}\boldsymbol{\gamma}_{k}\|^{2},\tag{S-1'}  \label{S-1'} \\
         \boldsymbol{c}_{k+1}&\in \mathop{\mathrm{argmin}}_{\boldsymbol{c}\in \mathbb{R}^{N}}\ G(\boldsymbol{c})+\frac{\beta}{2}\| \boldsymbol{\alpha}_{k+1}-\mathcal{A}_{K}^{2m}\boldsymbol{c}^{2m-1}+\frac{1}{\beta}\boldsymbol{\gamma}_{k}\|^{2} \tag{S-2'} \label{S-2'}.
      \end{align}
By definition, it is easy to check that (\ref{S-1'}) is lower semi-continuous and (\ref{S-2'}) is continuous. Moreover, (\ref{S-1'}) and (\ref{S-2'}) are coercive (see \cite[Definition 2.13]{Beck2017}), Weierstrass Theorem \cite[Theorem 2.14]{Beck2017} assures that (\ref{S-1'}) and (\ref{S-2'}) both have a solution. As a consequence, the splitting method above is well-defined and an infinite iterative sequence $\{(\boldsymbol{\alpha}_{k},\boldsymbol{c}_{k},\boldsymbol{\gamma}_{k},s_{k})\}$ is generated. Also, $\{s_{k}\}$ can be seen as an infinite iterative sequence to approximate the minimizer of Optimization (\ref{2.1}). 

    Next, we discuss how to solve subproblems (\ref{S-1'}) and (\ref{S-2'}). As for (\ref{S-1'}), since $F$ and $\|\cdot\|^{2}$ can be split of the variable into subvectors, we equivalently transfer an Optimization in $\mathbb{R}^{N}$ to some Optimizations in $\mathbb{R}$, that is, for $i=1,...,N$,
    \begin{equation}\tag{S-1''}\label{S-1''}
       (\boldsymbol{\alpha}_{k+1})_{i}\in \mathop{\mathrm{argmin}}_{\alpha_{i}}\ \frac{L(\boldsymbol{x}_{i},y_{i},\alpha_{i})}{N} +\frac{\beta}{2}\left(\alpha_{i}-(\mathcal{A}_{K}^{2m}(\boldsymbol{c}_{k})^{2m-1})_{i}+\frac{1}{\beta}(\boldsymbol{\gamma}_{k})_{i}\right)^{2}.
    \end{equation}
In other words, we solve Optimization (\ref{S-1'}) in $\mathbb{R}^{N}$ by breaking it into $N$ Optimizations (\ref{S-1''}) in $\mathbb{R}$, each of them easier to handle. For the general lower semi-continuous loss function $L$, the minimizer set of (\ref{S-1''}) may not be a singleton. In this case, we choose one of the elements in the minimizer set as $(\boldsymbol{\alpha}_{k+1})_{i},\ i=1,2,...,N$. 

    As for (\ref{S-2'}), it is easy to see that (\ref{S-2'}) is nonconvex and continuously differentiable. By definition of (\ref{S-2'}), (\ref{2.3}) and (\ref{2.4}), we have
    \begin{align*}
     &\nabla(G+\frac{\beta}{2}\| \boldsymbol{\alpha}_{k+1}-\mathcal{A}_{K}^{2m}(\cdot)^{2m-1}+\frac{1}{\beta}\boldsymbol{\gamma}_{k}\|^{2})(\boldsymbol{c}) \\
    =&2m\lambda \mathcal{A}_{K}^{2m}\boldsymbol{c}^{2m-1}-\beta(2m-1)\mathcal{A}_{K}^{2m}\boldsymbol{c}^{2m-2}\cdot(\boldsymbol{\alpha}_{k+1}-\mathcal{A}_{K}^{2m}\boldsymbol{c}^{2m-1}+\frac{1}{\beta}\boldsymbol{\gamma}_{k}) \notag \\
    =&\beta(2m-1)\mathcal{A}_{K}^{2m}\boldsymbol{c}^{2m-2}\cdot \left(\mathcal{A}_{K}^{2m}\boldsymbol{c}^{2m-1}+ \frac{2m\lambda }{(2m-1)\beta}\boldsymbol{c}-(\boldsymbol{\alpha}_{k+1}+\frac{1}{\beta}\boldsymbol{\gamma}_{k})\right) \notag.
    \end{align*}
Next we show that the solution of the following tensor equation
    \begin{equation}\tag{S-2''} \label{S-2''}
       \mathcal{A}_{K}^{2m}\boldsymbol{c}^{2m-1}+ \frac{2m\lambda}{(2m-1)\beta}\boldsymbol{c}=\boldsymbol{\alpha}_{k+1}+\frac{1}{\beta}\boldsymbol{\gamma}_{k}.
    \end{equation}    
is a minimizer of (\ref{S-2'}) by comparing the function values. We denote 
    $$
      H(\boldsymbol{c}):=\frac{1}{2m} \mathcal{A}_{K}^{2m}\boldsymbol{c}^{2m}+\frac{m\lambda}{(2m-1)\beta} \|\boldsymbol{c}\|^{2}-(\boldsymbol{\alpha}_{k+1}+\frac{1}{\beta}\boldsymbol{\gamma}_{k})^{T}\boldsymbol{c}. 
    $$
Then the solution of (\ref{S-2''}) is a stationary point of $H$. Also, (\ref{2.4}) shows that
    $$
       \nabla^{2} H(\boldsymbol{c})=(2m-1)\mathcal{A}_{K}^{2m}\boldsymbol{c}^{2m-2}+\frac{2m\lambda}{(2m-1)\beta} I,\ \forall \boldsymbol{c}\in \mathbb{R}^{N},
    $$
where $I$ denotes the identity matrix of $N$-dimensional. Since $\mathcal{A}_{K}^{2m}\boldsymbol{c}^{2m-2}$ is symmetric and positive definite, \cite[Proposition 1.1.10 (ii)]{Bertsekas2009} guarantees that $H$ is strictly convex on $\mathbb{R}^{N}$, which ensures that (\ref{S-2''}) has a unique solution $\bar{\boldsymbol{c}}$. By (\ref{S-2''}), (\ref{2.3}) and (\ref{2.4}), it follows that for any $\boldsymbol{c}\in \mathbb{R}^{N}$, 
    \begin{align*}
     &G(\boldsymbol{\boldsymbol{c}})+\tfrac{\beta}{2}\| \boldsymbol{\alpha}_{k+1}-\mathcal{A}_{K}^{2m}\boldsymbol{c}^{2m-1}+\tfrac{1}{\beta}\boldsymbol{\gamma}_{k}\|^{2}-G(\bar{\boldsymbol{c}})-\tfrac{\beta}{2}\| \boldsymbol{\alpha}_{k+1}-\mathcal{A}_{K}^{2m}\bar{\boldsymbol{c}}^{2m-1}+\tfrac{1}{\beta}\boldsymbol{\gamma}_{k}\|^{2} \\
   =&\lambda\mathcal{A}_{K}^{2m}\boldsymbol{c}^{2m}+\tfrac{\beta}{2}\| \mathcal{A}_{K}^{2m}\bar{\boldsymbol{c}}^{2m-1}+\tfrac{2m\lambda}{(2m-1)\beta}\bar{\boldsymbol{c}}-\mathcal{A}_{K}^{2m}\boldsymbol{c}^{2m-1}\|^{2}-\lambda\mathcal{A}_{K}^{2m}\bar{\boldsymbol{c}}^{2m}-\tfrac{\beta}{2}\| \tfrac{2m\lambda}{(2m-1)\beta}\bar{\boldsymbol{c}}\|^{2} \\
   =&\lambda\mathcal{A}_{K}^{2m}\boldsymbol{c}^{2m}+\tfrac{\lambda}{2m-1}\mathcal{A}_{K}^{2m}\bar{\boldsymbol{c}}^{2m}-\tfrac{2m\lambda}{2m-1}(\mathcal{A}_{K}^{2m}\boldsymbol{c}^{2m-1})^{T}\bar{\boldsymbol{c}}+\tfrac{\beta}{2}\| \mathcal{A}_{K}^{2m}\bar{\boldsymbol{c}}^{2m-1}-\mathcal{A}_{K}^{2m}\boldsymbol{c}^{2m-1}\|^{2}.
    \end{align*}
Since $\mathcal{A}_{K}^{2m}\boldsymbol{c}^{2m-2}$ is a positive definite matrix, \cite[Proposition 1.1.7 (a) and 1.1.10 (i)]{Bertsekas2009} and (\ref{2.4}) assure that $\boldsymbol{c}\mapsto \frac{\lambda}{2m-1}\mathcal{A}_{K}^{2m}\boldsymbol{c}^{2m}$ is convex and
    $$
    \frac{\lambda}{2m-1}\mathcal{A}_{K}^{2m}\bar{\boldsymbol{c}}^{2m}\geq \frac{\lambda}{2m-1}\mathcal{A}_{K}^{2m}\boldsymbol{c}^{2m}+(\frac{2m\lambda}{2m-1} \mathcal{A}_{K}^{2m}\boldsymbol{c}^{2m-1})^{T}(\bar{\boldsymbol{c}}-\boldsymbol{c}).
    $$
By rearranging terms, we see that
    $$
    \lambda\mathcal{A}_{K}^{2m}\boldsymbol{c}^{2m}+\frac{\lambda}{2m-1}\mathcal{A}_{K}^{2m}\bar{\boldsymbol{c}}^{2m}-\frac{2m\lambda}{2m-1}(\mathcal{A}_{K}^{2m}\boldsymbol{c}^{2m-1})^{T}\bar{\boldsymbol{c}}\geq 0.
    $$
Since $\frac{\beta}{2}\| \mathcal{A}_{K}^{2m}\bar{\boldsymbol{c}}^{2m-1}-\mathcal{A}_{K}^{2m}\boldsymbol{c}^{2m-1}\|^{2}\geq 0$, we conclude that for any $\boldsymbol{c}\in \mathbb{R}^{N}$, 
    $$
       G(\boldsymbol{\boldsymbol{c}})+\frac{\beta}{2}\| \boldsymbol{\alpha}_{k+1}-\mathcal{A}_{K}^{2m}\boldsymbol{c}^{2m-1}+\frac{1}{\beta}\boldsymbol{\gamma}_{k}\|^{2} \geq G(\bar{\boldsymbol{c}})+\frac{\beta}{2}\| \boldsymbol{\alpha}_{k+1}-\mathcal{A}_{K}^{2m}\bar{\boldsymbol{c}}^{2m-1}+\frac{1}{\beta}\boldsymbol{\gamma}_{k}\|^{2},
    $$
which ensures that $\boldsymbol{c}_{k+1}=\bar{\boldsymbol{c}}$. Next we consider to use Newton method for (\ref{S-2''}), whose convergence can be guaranteed by \cite[10.2.2 Newton Attraction Theorem]{Ortega1970}.
    
    When $\boldsymbol{\alpha}_{k+1}$ and $\boldsymbol{c}_{k+1}$ is acquired, we obtain $\boldsymbol{\gamma}_{k+1}$ by (\ref{S-3}). However, we have a simpler one that accomplishes the same goal. Substituting (\ref{S-3}) into (\ref{S-2''}), we have
    \begin{equation}\tag{S-3'} \label{S-3'}
        \boldsymbol{\gamma}_{k+1}=\frac{2m\lambda}{2m-1} \boldsymbol{c}_{k+1}.
    \end{equation}

    In conclusion, we present the following splitting method for Optimization (\ref{2.1}) with lower semi-continuous loss function below.
    
    \begin{algorithm}[H]
    \caption{\footnotesize Splitting Method for Optimization (\ref{2.1}) with Lower Semi-continuous Loss Function} 
    \begin{algorithmic}
    \footnotesize
      \STATE \textbf{input:} initial value $(\boldsymbol{\alpha}^{0},\boldsymbol{c}^{0}, \boldsymbol{\gamma}^{0})$, the training data $D$, lower semi-continuous loss function $L$, kernel $K$, $\lambda>0$, $\beta>0$, $m\in \mathbb{N}$ and stopping threshold $\varepsilon_{1},\varepsilon_{2}>0$.

      \STATE \textbf{for} $k=0,1,2,...$ \textbf{do}

      \qquad (1) Solve (\ref{S-1''}) and take the output as $\boldsymbol{\alpha}_{k+1}$.

      \qquad (2) \textbf{input:} initial value $\boldsymbol{c}_{k,0}=\boldsymbol{c}_{k}$

      \qquad \qquad \textbf{for} $j=0,1,2,...$ \textbf{do}

      \qquad \qquad \qquad (2-1) Solve the linear system 
      $$
         \nabla^{2} H(\boldsymbol{c}_{k,j})\cdot \boldsymbol{d}= -\nabla H(\boldsymbol{c}_{k,j})
      $$
        
      \qquad \qquad \qquad and take the solution $\boldsymbol{d}_{k,j}$.

      \qquad \qquad \qquad (2-2) Set $\boldsymbol{c}_{k,j+1}\leftarrow \boldsymbol{c}_{k,j}+\boldsymbol{d}_{k,j}$.

      \qquad \qquad \qquad \textbf{if} $\|\boldsymbol{d}_{k,j}\|<\varepsilon_{1}$ \textbf{then} stop.

      \qquad \qquad \textbf{end for}

      \qquad \qquad \textbf{output:} The approximate solution $\boldsymbol{c}_{k,j+1}$ as $\boldsymbol{c}_{k+1}$.

      \qquad (3) Set $\boldsymbol{\gamma}_{k+1}\leftarrow \frac{2m\lambda}{2m-1} \boldsymbol{c}_{k+1}$.

      \qquad (4) Set $s_{k+1}\leftarrow \sum\limits_{n\in \mathbb{N}}\left(\boldsymbol{\Phi}_{n}^{T}\boldsymbol{c}_{k+1}\right)^{2m-1}\phi_{n}$.

      \qquad \textbf{if} $\|\boldsymbol{\alpha}_{k+1}-\mathcal{A}_{K}^{2m}(\boldsymbol{c}_{k+1})^{2m-1}\|< \varepsilon_{2}$ \textbf{then} stop.

      \STATE \textbf{end for}

      \STATE \textbf{output:} The proposed solution is $s_{k+1}$ obtained after $k+1$ iterations.
    \end{algorithmic}
    \label{alg:ADMM}
    \end{algorithm}

    In the next section, we verify that under some mild assumptions, $\{s_{k}\}$ is globally convergent to a stationary point of Optimization (\ref{2.1}). Hence, it is better to solve Optimization (\ref{2.1}) repeatedly by selecting some initial values randomly and choosing the minimizer of these outputs as the approximate solution $s_{D}^{\frac{2m}{2m-1}}:X\to \mathbb{R}$ of Optimization (\ref{2.1}). At last, we build the SVM $\mathcal{R}s_{D}^{\frac{2m}{2m-1}}:X\rightarrow Y$ to make the prediction according to task requirement. 
    
    \section{Convergence Analysis}
    \label{sec:4}
    In this section, we investigate the convergence of $\{s_{k}\}$ inspired from the work \cite{Guo2016Convergence,Li2015Global} and use similar line of arguments therein. 

    \subsection{Assumptions}
    To ensure the convergence, we need the additional conditions of Optimization (\ref{2.1}) which we describe below. First, we denote the space consisting of all symmetric matrix of $N$-dimensional as $\mathbb{S}^{N\times N}$. It is easy to check that $\mathbb{S}^{N\times N}$ is a $\frac{N(N+1)}{2}$-dimensional space. For $\{\boldsymbol{\Phi}_{1},...,\boldsymbol{\Phi}_{n},...\}\subseteq \mathbb{R}^{N}$, we have that $\{\boldsymbol{\Phi}_{1}\boldsymbol{\Phi}_{1}^{T},...,\boldsymbol{\Phi}_{n}\boldsymbol{\Phi}_{n}^{T},...\} \subseteq \mathbb{S}^{N\times N}$. In conclusion, we have the following assumptions.

   \begin{assumption}\label{Assumption:4.1}
      For Optimization (\ref{2.1}), the following condition hold

      \text{\normalfont{(i)}} $L$ is a lower semi-continuous and subanalytic loss function.

      \text{\normalfont{(ii)}} For any $\boldsymbol{x}\in X$ and $y\in \mathbb{R}$, $L(\boldsymbol{x},y,\cdot)$ is continuous at $0$.

      \text{\normalfont{(iii)}} For any $\boldsymbol{x}\in X$ and $y\in \mathbb{R}$, $0$ is not the stationary point of $L(\boldsymbol{x},y,\cdot)$  in the sense of limiting subdifferential, that is, $0\notin \partial L(\boldsymbol{x},y,0)$. \text{\normalfont{(see \cite[Definition 8.3]{Rockafellar1998})}}

      \text{\normalfont{(iv)}} $\mathrm{rank}(\boldsymbol{\Phi}_{1}\boldsymbol{\Phi}_{1}^{T},...,\boldsymbol{\Phi}_{n}\boldsymbol{\Phi}_{n}^{T},...)=\frac{N(N+1)}{2}$.
   \end{assumption}

    Under Assumption \ref{Assumption:4.1}, we see that (i), (ii) and (iii) are the additional conditions of loss function and (iv) is the additional condition of RKBS. Subanalytic functions are quite wide, including semi-algebraic, analytic and semi-analytic functions (see \cite[6.6 Analytic Problems]{Finite}). More precisely, polynomial functions and piecewise polynomial functions are subanalytic functions. However, subanalyticity does not even imply continuity. Specially, some margin-based loss functions (see \cite[Section 2.3]{MR2450103}) satisfy Assumption \ref{Assumption:4.1} (i), (ii) and (iii), such as the least square loss, the Hinge loss, the truncated least squares loss, logistic loss and so on. For the given training data, we can find a suitable RKBS and verify Assumption \ref{Assumption:4.1} (iv). In numerical experiment, for all $n\in \mathbb{N}$, since $\boldsymbol{\Phi}_{n}\boldsymbol{\Phi}_{n}^{T}$ is symmetric, we can take the upper triangle part and vectorize them as row vectors. If we find a $N(N+1)/2\times N$-dimensional  nonsingular submatrix consisting of these row vectors, then we can verify the Assumption \ref{Assumption:4.1} (iv) to ensure the convergence before using Algorithm \ref{alg:ADMM}.  

    In the rest of this subsection, we discuss what conclusions can be drawn under Assumption \ref{Assumption:4.1}. By definition of $\mathcal{L}_{\beta}$ and Assumption \ref{Assumption:4.1} (i), we check that $\mathcal{L}_{\beta}$ is lower semi-continuous. Since $F$, $G$ and $\frac{\beta}{2}\|\boldsymbol{\alpha}-\mathcal{A}_{K}^{2m}\boldsymbol{c}^{2m-1}\|^{2}$ are nonnegative and subanalytic and $\boldsymbol{\gamma}^{T}(\boldsymbol{\alpha}-\mathcal{A}_{K}^{2m}\boldsymbol{c}^{2m-1})$ is subanalytic and bounded for any bounded set in $\mathbb{R}^{3N}$, \cite[(I.2.1.9)]{1997Geometry} shows that $\mathcal{L}_{\beta}$ is subanalytic. Moreover, \cite{Bolte2007,Bolte2,Xu2013} assures that $\mathcal{L}_{\beta}$ is a KL function on $\mathbb{R}^{3N}$, that is, $\mathcal{L}_{\beta}$ has KL property at each point in $\mathbb{R}^{3N}$ (see \cite[Section 2.4]{Bolte2014Proximal}). The KL property of $\mathcal{L}_{\beta}$ plays a cruical role in estimating the error bounds of the iterative sequence.

    Next we verify that $(\boldsymbol{0},\boldsymbol{0})$ is not the subsequential limit of $\{(\boldsymbol{c}_{k},\boldsymbol{c}_{k+1})\}$ under Assumption \ref{Assumption:4.1} (ii) and (iii). Assume that there exists a subsequence $\{\boldsymbol{c}_{k_{j}}\}$ such that $(\boldsymbol{c}_{k_{j}},\boldsymbol{c}_{k_{j}+1})\to (\boldsymbol{0},\boldsymbol{0})$. From (\ref{S-3}) and (\ref{S-3'}), we also have that $\boldsymbol{\gamma}_{k_{j}}\to \boldsymbol{0}$, $\boldsymbol{\gamma}_{k_{j}+1}\to \boldsymbol{0}$ and $\boldsymbol{\alpha}_{k_{j}+1}\to \boldsymbol{0}$. Moreover, from the optimality condition of (\ref{S-1}), the iterates generated satisfy
    $$
      \boldsymbol{0}\in \partial F(\boldsymbol{\alpha}_{k_{j}+1})+\boldsymbol{\gamma}_{k_{j}}+\beta(\boldsymbol{\alpha}_{k_{j}+1}-\mathcal{A}_{K}^{2m}(\boldsymbol{c}_{k_{j}})^{2m-1}).
    $$
By Assumption \ref{Assumption:4.1} (ii), we have that $F$ is continuous at $\boldsymbol{0}$, which ensures that $F(\boldsymbol{\alpha}_{k_{j}+1})\to F(\boldsymbol{0})$. Therefore, passing to the limit along $\{(\boldsymbol{\alpha}_{k_{j}},\boldsymbol{c}_{k_{j}},\boldsymbol{\gamma}_{k_{j}})\}$, \cite[proposition 8.7]{Rockafellar1998} shows that
    $$
       \boldsymbol{0}\in \partial F(\boldsymbol{0}).
    $$
Since $F$ can be split of the variable into subvectors, Assumption \ref{Assumption:4.1} (iii) and \cite[proposition 10.5]{Rockafellar1998} and  \cite[D. Rescaling]{Rockafellar1998} assure that
    $$
       \boldsymbol{0}\notin \partial F(\boldsymbol{0})=\left(\frac{1}{N}\partial L(\boldsymbol{x}_{1},y_{1},0)\right)\times\cdots \times \left(\frac{1}{N}\partial L(\boldsymbol{x}_{N},y_{N},0)\right).
    $$
Clearly, the two relations above are contradiction. Hence, $(\boldsymbol{0},\boldsymbol{0})$ is not the subsequential limit of $\{(\boldsymbol{c}_{k},\boldsymbol{c}_{k+1})\}$, which means that there exists a neighbourhood of $(\boldsymbol{0},\boldsymbol{0})$ and $k_{0}\in \mathbb{N}$ such that the neighbourhood does not have any point of $\{(\boldsymbol{c}_{k},\boldsymbol{c}_{k+1})\}$ when $k>k_{0}$. Furthermore, we observe that if $\{\boldsymbol{c}_{k}\}$ is convergent, then $\{\boldsymbol{c}_{k}\}$ can not converge to $\boldsymbol{0}$.  

    Under Assumption \ref{Assumption:4.1} (iv), we obtain the symmetry and strictly positive definiteness of $\mathcal{A}_{K}^{2m}\boldsymbol{c}^{2m-2}$ for any $\boldsymbol{c}\neq \boldsymbol{0}$. From Subsection \ref{sec:2.3}, $\mathcal{A}_{K}^{2m}\boldsymbol{c}^{2m-2}$ is symmetric and positive definite. If $\boldsymbol{d}^{T}\left(\mathcal{A}_{K}^{2m}\boldsymbol{c}^{2m-2}\right)\boldsymbol{d}=0$, then for any $n\in \mathbb{N}$, $(\boldsymbol{\Phi}_{n}^{T}\boldsymbol{c})(\boldsymbol{\Phi}_{n}^{T}\boldsymbol{d})=(\boldsymbol{\Phi}_{n}\boldsymbol{\Phi}_{n}^{T})\cdot(\boldsymbol{c}\boldsymbol{d}^{T})=0$. By Assumption \ref{Assumption:4.1} (iv), the following system
    $$
    (\boldsymbol{\Phi}_{1}\boldsymbol{\Phi}_{1}^{T},...,\boldsymbol{\Phi}_{n}\boldsymbol{\Phi}_{n}^{T},...)^{T}\cdot (\boldsymbol{c}\boldsymbol{d}^{T})=(0,...,0,...)^{T}
    $$
only has zero solution, that is, $\boldsymbol{c}\boldsymbol{d}^{T}=O$, where $O$ denotes zero matrix. Since $\boldsymbol{c}\neq \boldsymbol{0}$, there exists at least one $c_{i}\neq 0$. Since $c_{i}d_{j}=0,\ j=1,2,...,N$, it is easy to check that $d_{1}=d_{2}=...=d_{N}=0$, that is, $\boldsymbol{d}=\boldsymbol{0}$. Hence, for any $\boldsymbol{c}\neq \boldsymbol{0}$ and $\boldsymbol{d}\neq \boldsymbol{0}$, we see that
    $$
       \boldsymbol{d}^{T}\left(\mathcal{A}_{K}^{2m}\boldsymbol{c}^{2m-2}\right)\boldsymbol{d}>0.
    $$
that is, $\mathcal{A}_{K}^{2m}\boldsymbol{c}^{2m-2}$ is symmetric and strictly positive definite when $\boldsymbol{c}\neq\boldsymbol{0}$. In particular, $\mathcal{A}_{K}^{2}$ is a symmetric and strictly positive definite matrix. Notably, $\mathcal{A}_{K}^{2m}\boldsymbol{0}^{2m-2}$ is a zero matrix of $N$-dimensional.

    Moreover, by definition, it is easy to check that $\Psi(f)=\sum\limits_{n\in \mathbb{N}}\left(\boldsymbol{\Phi}_{n}^{T}\boldsymbol{c}\right)\phi_{n}\mapsto \boldsymbol{c}$ is a linear operator from $\mathrm{span}\{K(\boldsymbol{x}_{1},\cdot),...,K(\boldsymbol{x}_{N},\cdot)\}$ to $\mathbb{R}^{N}$. For any $f\in \Delta_{K}^{\frac{2m}{2m-1}}(X)$, we assume that there exists $\boldsymbol{c}, \boldsymbol{d}\in \mathbb{R}^{N}$ such that
    $$
    \Psi(f)=\sum\limits_{n\in \mathbb{N}}\left(\boldsymbol{\Phi}_{n}^{T}\boldsymbol{c}\right)\phi_{n}=\sum\limits_{n\in \mathbb{N}}\left(\boldsymbol{\Phi}_{n}^{T}\boldsymbol{d}\right)\phi_{n}.
    $$
Hence, $\sum\limits_{n\in \mathbb{N}}\left(\boldsymbol{\Phi}_{n}^{T}(\boldsymbol{c}-\boldsymbol{d})\right)\phi_{n}=0
$, that is, 
    $$
      (\boldsymbol{\Phi}_{1},...,\boldsymbol{\Phi}_{N},...)^{T}(\boldsymbol{c}-\boldsymbol{d})=(0,...,0,...)^{T}.
    $$
Since $\mathcal{A}_{K}^{2}$ is symmetric and strictly positive definite, $\mathrm{rank}(\mathcal{A}_{K}^{2})=N$. As 
    $$
    \mathrm{rank}(\mathcal{A}_{K}^{2})=\mathrm{rank}((\boldsymbol{\Phi}_{1},...,\boldsymbol{\Phi}_{n},...)(\boldsymbol{\Phi}_{1},...,\boldsymbol{\Phi}_{n},...)^{T})\leq \mathrm{rank}((\boldsymbol{\Phi}_{1},...,\boldsymbol{\Phi}_{n},...))\leq N,
    $$
one can see $\mathrm{rank}((\boldsymbol{\Phi}_{1},...,\boldsymbol{\Phi}_{n},...))=N$ and we conclude that $\boldsymbol{c}=\boldsymbol{d}$. Thus, $\Psi(f)\mapsto \boldsymbol{c}$ is a linear operator from $\mathrm{span}\{K(\boldsymbol{x}_{1},\cdot),...,K(\boldsymbol{x}_{N},\cdot)\}$ onto $\mathbb{R}^{N}$ and thus an isomorphism by \cite[1.4.15 Theorem]{Megginson1998}.

    \subsection{Convergence of Splitting Method}
    In this subsection, based on the Assumption  \ref{Assumption:4.1}, we introduce the convergence of $\{s_{k}\}$. 

    \begin{theorem}\label{Theorem:4.2}
       Suppose that Assumption \ref{Assumption:4.1} holds. There exists $\beta_{min}>0$ such that $\{s_{k}\}$ converges to a stationary point of Optimization (\ref{2.1}) when $\beta>\beta_{min}$.
    \end{theorem}

    For convenience, we will abbreviate $\Psi_{K}^{\frac{2m}{2m-1}}$ as $\Psi$ for simplicity in this subsection. The main idea of the proof is to use the property of $\{\mathcal{L}_{\beta}(\boldsymbol{\alpha}_{k},\boldsymbol{c}_{k},\boldsymbol{\gamma}_{k})\}$ to deduce the convergence of $\left\{\Psi(s_{k})\right\}$. Since $\Psi$ is a homeomorphism, we verify the convergence of $\{s^{k}\}$. Before presenting our main result in this section, we introduce a useful inequality which plays a cruical role in estimating the error bounds of $\left\{\Psi(s_{k})\right\}$.

    \begin{lemma}\label{Lemma:4.3} 
    Suppose that Assumption \ref{Assumption:4.1} holds. There exists $\beta_{min}>0$, $\zeta_{1}>0$ and $k_{0}\in \mathbb{N}$ such that when $\beta>\beta_{min}$ and $k>k_{0}$, the following descent inequality holds 
    $$
    \zeta_{1}\|\Psi(s_{k})-\Psi(s_{k+1})\|_{\mathcal{B}_{K}^{2m}(X)}^{2}\leq \mathcal{L}_{\beta}(\boldsymbol{\alpha}_{k},\boldsymbol{c}_{k},\boldsymbol{\gamma}_{k})-\mathcal{L}_{\beta}(\boldsymbol{\alpha}_{k+1},\boldsymbol{c}_{k+1},\boldsymbol{\gamma}_{k+1}).
    $$
    \end{lemma}

    \begin{proof}
       From (\ref{S-1}), we know that $\boldsymbol{\alpha}_{k+1}$ is the minimizer of $\mathcal{L}_{\beta}(\boldsymbol{\alpha},\boldsymbol{c}_{k},\boldsymbol{\gamma}_{k})$, therefore
     \begin{equation}\label{4.1}
        0\leq \mathcal{L}_{\beta}(\boldsymbol{\alpha}_{k},\boldsymbol{c}_{k},\boldsymbol{\gamma}_{k})-\mathcal{L}_{\beta}(\boldsymbol{\alpha}_{k+1},\boldsymbol{c}_{k},\boldsymbol{\gamma}_{k}).
     \end{equation}

    Similarily from the definition of $\mathcal{L}_{\beta}$, (\ref{S-3}) and rearranging terms, we see that
     \begin{align}\label{4.2}
       &\mathcal{L}_{\beta}(\boldsymbol{\alpha}_{k+1},\boldsymbol{c}_{k},\boldsymbol{\gamma}_{k})-\mathcal{L}_{\beta}(\boldsymbol{\alpha}_{k+1},\boldsymbol{c}_{k+1},\boldsymbol{\gamma}_{k}) \notag \\ 
      =&G(\boldsymbol{c}_{k})-G(\boldsymbol{c}_{k+1})-(\boldsymbol{\gamma}_{k})^{T}(\mathcal{A}_{K}^{2m}(\boldsymbol{c}_{k})^{2m-1}-\mathcal{A}_{K}^{2m}(\boldsymbol{c}_{k+1})^{2m-1}) \notag \\
        &+\frac{\beta}{2}\|\boldsymbol{\alpha}_{k+1}-\mathcal{A}_{K}^{2m}(\boldsymbol{c}_{k})^{2m-1}\|^{2}-\frac{\beta}{2}\|\boldsymbol{\alpha}_{k+1}-\mathcal{A}_{K}^{2m}(\boldsymbol{c}_{k+1})^{2m-1}\|^{2}. \\
      =&G(\boldsymbol{c}_{k})-G(\boldsymbol{c}_{k+1})-(\boldsymbol{\gamma}_{k+1})^{T}(\mathcal{A}_{K}^{2m}(\boldsymbol{c}_{k})^{2m-1}-\mathcal{A}_{K}^{2m}(\boldsymbol{c}_{k+1})^{2m-1}) \notag \\
 &+\frac{\beta}{2}\|\mathcal{A}_{K}^{2m}(\boldsymbol{c}_{k})^{2m-1}-\mathcal{A}_{K}^{2m}(\boldsymbol{c}_{k+1})^{2m-1}\|^{2}. \notag
    \end{align}
 it follows that for any $\boldsymbol{c}\in \mathbb{R}^{N}$, 
    \begin{align*}
     &G(\boldsymbol{c})-G(\boldsymbol{c}_{k+1})-(\boldsymbol{\gamma}_{k+1})^{T}(\mathcal{A}_{K}^{2m}\boldsymbol{c}^{2m-1}-\mathcal{A}_{K}^{2m}(\boldsymbol{c}_{k+1})^{2m-1}) \\
   =&\lambda \mathcal{A}_{K}^{2m}(\boldsymbol{c}_{k})^{2m}-\lambda\mathcal{A}_{K}^{2m}(\boldsymbol{c}_{k+1})^{2m}-\tfrac{2m\lambda}{2m-1} (\mathcal{A}_{K}^{2m}(\boldsymbol{c}_{k})^{2m-1})^{T}\boldsymbol{c}_{k+1}+\tfrac{2m\lambda}{2m-1}\mathcal{A}_{K}^{2m}(\boldsymbol{c}_{k+1})^{2m} \\
   =&\lambda \mathcal{A}_{K}^{2m}(\boldsymbol{c}_{k})^{2m}+\tfrac{\lambda}{2m-1}\mathcal{A}_{K}^{2m}(\boldsymbol{c}_{k+1})^{2m}-\tfrac{2m\lambda}{2m-1} (\mathcal{A}_{K}^{2m}(\boldsymbol{c}_{k})^{2m-1})^{T}\boldsymbol{c}_{k+1}.
    \end{align*}
Since $\mathcal{A}_{K}^{2m}\boldsymbol{c}^{2m-2}$ is positive definite, \cite[Proposition 1.1.7 (a) and 1.1.10 (i)]{Bertsekas2009} and (\ref{2.4}) assure that $\boldsymbol{c}\mapsto \frac{\lambda}{2m-1}\mathcal{A}_{K}^{2m}\boldsymbol{c}^{2m}$ is convex, and 
     $$
    \frac{\lambda}{2m-1}\mathcal{A}_{K}^{2m}(\boldsymbol{c}_{k+1})^{2m}\geq \frac{\lambda}{2m-1}\mathcal{A}_{K}^{2m}(\boldsymbol{c}_{k})^{2m}+(\frac{2m\lambda}{2m-1} \mathcal{A}_{K}^{2m}(\boldsymbol{c}_{k})^{2m-1})^{T}(\boldsymbol{c}_{k+1}-\boldsymbol{c}_{k}).
     $$
Combining (\ref{2.3}) with two relations above, it follows that
    \begin{equation}\label{4.3}
       G(\boldsymbol{c}_{k})-G(\boldsymbol{c}_{k+1})-(\boldsymbol{\gamma}_{k+1})^{T}(\mathcal{A}_{K}^{2m}(\boldsymbol{c}_{k})^{2m-1}-\mathcal{A}_{K}^{2m}(\boldsymbol{c}_{k+1})^{2m-1})\geq 0.
    \end{equation}
From Assumption \ref{Assumption:4.1} (ii), (iii) and (iv), there exists $\nu>0$ and $k_{0}\in \mathbb{N}$ such that for any $k>k_{0}$, $\mathcal{A}_{K}^{2m}(\boldsymbol{c}_{k+1})^{2m-2}+\mathcal{A}_{K}^{2m}(\boldsymbol{c}_{k})^{2m-2}-2\nu I$ is strictly positive definite. Thus, 
      \begin{align*}
        &(\boldsymbol{c}_{k}-\boldsymbol{c}_{k+1})^{T}(\mathcal{A}_{K}^{2m}(\boldsymbol{c}_{k})^{2m-1}-\mathcal{A}_{K}^{2m}(\boldsymbol{c}_{k+1})^{2m-1})  \\
       =\ &(\boldsymbol{c}_{k}-\boldsymbol{c}_{k+1})^{T}\sum_{n\in \mathbb{N}} \left(\left(\boldsymbol{\Phi}_{n}^{T}\boldsymbol{c}_{k}\right)^{2m-1}-\left(\boldsymbol{\Phi}_{n}^{T}\boldsymbol{c}_{k+1}\right)^{2m-1}\right)\boldsymbol{\Phi}_{n} \notag \\
       =\ &(\boldsymbol{c}_{k}-\boldsymbol{c}_{k+1})^{T}\left(\sum_{n\in \mathbb{N}} \left(\sum_{j=0}^{2m-2} \left(\boldsymbol{\Phi}_{n}^{T}\boldsymbol{c}_{k}\right)^{2m-2-j}\left(\boldsymbol{\Phi}_{n}^{T}\boldsymbol{c}_{k+1}\right)^{j}\right)\boldsymbol{\Phi}_{n}\boldsymbol{\Phi}_{n}^{T}\right)(\boldsymbol{c}_{k}-\boldsymbol{c}_{k+1}) \notag \\
       \geq\ &\frac{1}{2}(\boldsymbol{c}_{k}-\boldsymbol{c}_{k+1})^{T}\left(\sum_{n\in \mathbb{N}} \left(\left(\boldsymbol{\Phi}_{n}^{T}\boldsymbol{c}_{k}\right)^{2m-2}+\left(\boldsymbol{\Phi}_{n}^{T}\boldsymbol{c}_{k+1}\right)^{2m-2}\right)\boldsymbol{\Phi}_{n}\boldsymbol{\Phi}_{n}^{T}\right)(\boldsymbol{c}_{k}-\boldsymbol{c}_{k+1}) \notag \\
       =\ & \frac{1}{2}(\boldsymbol{c}_{k}-\boldsymbol{c}_{k+1})^{T}\left(\mathcal{A}_{K}^{2m}(\boldsymbol{c}_{k})^{2m-2}+\mathcal{A}_{K}^{2m}(\boldsymbol{c}_{k+1})^{2m-2}\right)(\boldsymbol{c}_{k}-\boldsymbol{c}_{k+1}). \notag \\
       \geq\ &\nu \|\boldsymbol{c}_{k}-\boldsymbol{c}_{k+1}\|^{2}.
     \end{align*}
From Cauchy-Schwartz inequality, it follows that 
    $$
      \|\mathcal{A}_{K}^{2m}(\boldsymbol{c}_{k})^{2m-1}-\mathcal{A}_{K}^{2m}(\boldsymbol{c}_{k+1})^{2m-1}\|\geq \nu \|\boldsymbol{c}_{k}-\boldsymbol{c}_{k+1}\|.
    $$
Since $s_{k},s_{k+1}\in \Delta_{K}^{\frac{2m}{2m-1}}(X)$, then
    $$
       \Psi(s_{k})-\Psi(s_{k+1})=\sum_{n\in \mathbb{N}}\left(\boldsymbol{\Phi}_{n}^{T}(\boldsymbol{c}_{k}-\boldsymbol{c}_{k+1})\right)\phi_{n}\in \mathrm{span}\{K(\boldsymbol{x}_{1},\cdot),...,K(\boldsymbol{x}_{N},\cdot)\}
    $$
By Assumption \ref{Assumption:4.1} (iv), we see that $\Psi(f)\mapsto \boldsymbol{c}$ is an isomorphism and \cite[1.4.14 Proposition (a)]{Megginson1998} assures that there are $0<w_{1}\leq w_{2}$ such that
    \begin{equation}\label{4.4}
       w_{1}\|\Psi(s_{k})-\Psi(s_{k+1})\|_{\mathcal{B}_{K}^{2m}(X)} \leq \|\boldsymbol{c}_{k}-\boldsymbol{c}_{k+1}\|\leq w_{2}\|\Psi(s_{k})-\Psi(s_{k+1})\|_{\mathcal{B}_{K}^{2m}(X)}.
    \end{equation}
Hence, 
    \begin{equation}\label{4.5}
       \frac{\beta(\nu w_{1})^{2}}{2}\|\Psi(s_{k})-\Psi(s_{k+1})\|_{\mathcal{B}_{K}^{2m}(X)}^{2}\leq \frac{\beta}{2}\|\mathcal{A}_{K}^{2m}(\boldsymbol{c}_{k})^{2m-1}-\mathcal{A}_{K}^{2m}(\boldsymbol{c}_{k+1})^{2m-1}\|^{2}.
    \end{equation}
From (\ref{4.2}), (\ref{4.3}) and (\ref{4.5}), we have that
    \begin{equation}\label{4.6}
       \frac{\beta(\nu w_{1})^{2}}{2}\|\Psi(s_{k})-\Psi(s_{k+1})\|_{\mathcal{B}_{K}^{2m}(X)}^{2}\leq \mathcal{L}_{\beta}(\boldsymbol{\alpha}_{k+1},\boldsymbol{c}_{k},\boldsymbol{\gamma}_{k})-\mathcal{L}_{\beta}(\boldsymbol{\alpha}_{k+1},\boldsymbol{c}_{k+1},\boldsymbol{\gamma}_{k}). 
    \end{equation}

    Furthermore, from the definition of $\mathcal{L}_{\beta}$ and (\ref{S-3}), it follows that
    \begin{equation}\label{4.7}
        -\frac{1}{\beta}\|\boldsymbol{\gamma}_{k}-\boldsymbol{\gamma}_{k+1}\|^{2}=\mathcal{L}_{\beta}(\boldsymbol{\alpha}_{k+1},\boldsymbol{c}_{k+1},\boldsymbol{\gamma}_{k})-\mathcal{L}_{\beta}(\boldsymbol{\alpha}_{k+1},\boldsymbol{c}_{k+1},\boldsymbol{\gamma}_{k+1}).
    \end{equation}
We denote 
    $$
       \zeta_{1}=\frac{4m^{2}\lambda^{2}(w_{2})^{2}}{(2m-1)^{2}\beta},\ \ \beta_{min}:=\frac{4m\lambda w_{2}}{(2m-1)\nu w_{1}}.
    $$
Then $\zeta_{1}>0$ and $\beta_{min}>0$. Inserting (\ref{S-3'}) and (\ref{4.4}) into (\ref{4.7}), we have that
    \begin{equation}\label{4.8}
        -\zeta_{1}\|\Psi(s_{k})-\Psi(s_{k+1})\|_{\mathcal{B}_{K}^{2m}(X)}^{2}\leq \mathcal{L}_{\beta}(\boldsymbol{\alpha}_{k+1},\boldsymbol{c}_{k+1},\boldsymbol{\gamma}_{k})-\mathcal{L}_{\beta}(\boldsymbol{\alpha}_{k+1},\boldsymbol{c}_{k+1},\boldsymbol{\gamma}_{k+1}).
    \end{equation}
If $\beta>\beta_{min}$, then
    $$
       \frac{\beta \nu^{2}(w_{1})^{2}}{2}-\zeta_{1}=\frac{\beta^{2}\nu^{2}(w_{1})^{2}-\frac{8m^{2}\lambda^{2}(w_{2})^{2}}{(2m-1)^{2}}}{2\beta}>\frac{4m^{2}\lambda^{2}(w_{2})^{2}}{(2m-1)^{2}\beta}=\zeta_{1}>0.
    $$
From (\ref{4.1}), (\ref{4.6}) and (\ref{4.8}), if $k>k_{0}$, then the following descent inequality holds
    $$
    \zeta_{1}\|\Psi(s_{k})-\Psi(s_{k+1})\|_{\mathcal{B}_{K}^{2m}(X)}^{2}\leq \mathcal{L}_{\beta}(\boldsymbol{\alpha}_{k},\boldsymbol{c}_{k},\boldsymbol{\gamma}_{k})-\mathcal{L}_{\beta}(\boldsymbol{\alpha}_{k+1},\boldsymbol{c}_{k+1},\boldsymbol{\gamma}_{k+1}).
    $$
This proof is completed.
    \end{proof}

    \begin{corollary}\label{Corollary:4.4}
    Suppose that Assumption \ref{Assumption:4.1} holds. If $\beta>\beta_{min}$, then the sequence $\{\mathcal{L}_{\beta}(\boldsymbol{\alpha}_{k},\boldsymbol{c}_{k},\boldsymbol{\gamma}_{k})\}$ is convergent.
    \end{corollary}    

    \begin{proof}
    By Lemma \ref{Lemma:4.3}, it follows that $\{\mathcal{L}_{\beta}(\boldsymbol{\alpha}_{k},\boldsymbol{c}_{k},\boldsymbol{\gamma}_{k})\}$ is monotonically decreasing when $k>k_{0}$. Since $F(\boldsymbol{\alpha}_{k})\geq 0$, we have that
    \begin{equation}\label{4.9}
       \mathcal{L}_{\beta}(\boldsymbol{\alpha}^{k_{0}+1},\boldsymbol{c}^{k_{0}+1},\boldsymbol{\gamma}^{k_{0}+1})\geq \mathcal{L}_{\beta}(\boldsymbol{\alpha}_{k},\boldsymbol{c}_{k},\boldsymbol{\gamma}_{k}) \geq G(\boldsymbol{c}_{k})-\frac{2m^{2}\lambda^{2}}{(2m-1)^{2}\beta}\|\boldsymbol{c}_{k}\|^{2}. 
    \end{equation}
If $m=1$, then $\nu>0$ is its strong convexity parameter because of the symmetry and strictly positive definiteness of $\mathcal{A}_{K}^{2}$. Hence, \cite[Example 5.19 and Theorem 5.24 (iii)]{Beck2017} show that
    $$
       G(\boldsymbol{c})=\lambda (\boldsymbol{c}-\boldsymbol{0})^{T}(\mathcal{A}_{K}^{2}\boldsymbol{c}-\boldsymbol{0})\geq \lambda\nu \|\boldsymbol{c}\|^{2},
    $$
Since $\beta>\beta_{min}=\frac{4\lambda w_{2}}{\nu w_{1}}$ and $w_{2}\geq w_{1}$, it follows that for any $\boldsymbol{c}\in \mathbb{R}^{N}$, 
    $$
       G(\boldsymbol{c})-\frac{2\lambda^{2}}{\beta}\|\boldsymbol{c}\|^{2} \geq \lambda\nu \|\boldsymbol{c}\|^{2}-\frac{\lambda\nu w_{1}}{2w_{2}} \|\boldsymbol{c}\|^{2}\geq \frac{\lambda\nu}{2} \|\boldsymbol{c}\|^{2}.
    $$
Hence, the function $G(\boldsymbol{c})-\frac{2\lambda^{2}}{\beta}\|\boldsymbol{c}\|^{2}$ is continuous and coercive. Moreover, if $m>1$, then by definition, $G(\boldsymbol{c})-\frac{2m^{2}\lambda^{2}}{(2m-1)^{2}\beta}\|\boldsymbol{c}\|^{2}$ is also continuous and coercive. Thus, according to Weierstrass Theorem, the sequence $\{G(\boldsymbol{c}_{k})-\frac{2m^{2}\lambda^{2}}{(2m-1)^{2}\beta}\|\boldsymbol{c}_{k}\|^{2}\}$ is lower bounded, which means that $\{\mathcal{L}_{\beta}(\boldsymbol{\alpha}_{k},\boldsymbol{c}_{k},\boldsymbol{\gamma}_{k})\}$ is also lower bounded when $k>k_{0}$. Hence, \cite[Theorem 3.24]{Rudin1976Principles} shows the convergence of $\{\mathcal{L}_{\beta}(\boldsymbol{\alpha}_{k},\boldsymbol{c}_{k},\boldsymbol{\gamma}_{k})\}$. This proof is completed.
    \end{proof}

    For notational convenience, since $\{\mathcal{L}_{\beta}(\boldsymbol{\alpha}_{k},\boldsymbol{c}_{k},\boldsymbol{\gamma}_{k})\}$ is convergent, we denote the residual 
    $$
      r_{k}:=\mathcal{L}_{\beta}(\boldsymbol{\alpha}_{k},\boldsymbol{c}_{k},\boldsymbol{\gamma}_{k})-\lim_{k\rightarrow \infty}\mathcal{L}_{\beta}(\boldsymbol{\alpha}_{k},\boldsymbol{c}_{k},\boldsymbol{\gamma}_{k}).
    $$
When $k>k_{0}$, Lemma \ref{Lemma:4.3} can be rewritten as , 
    $$
     \zeta_{1}\|\Psi(s_{k})-\Psi(s_{k+1})\|_{\mathcal{B}_{K}^{2m}(X)}^{2}\leq r_{k}-r_{k+1}.
    $$
Moreover, we conclude that $\{r_{k}\}$ is monotonically decreasing when $k>k_{0}$ and converges to $0$.

    Now we prove Theorem \ref{Theorem:4.2} by Lemma \ref{Lemma:4.3} and Corollary \ref{Corollary:4.4}. 

     \begin{proof}[\bf{Proof of \ref{Theorem:4.2}}]
    We consider the following two cases of $\{r_{k}\}$:

    (I) If for $k$ sufficiently large, $r_{k}=r_{k+1}=0$, then Lemma \ref{Lemma:4.3} shows that $\Psi(s_{k})=\Psi(s_{k+1})$. Since $\Psi$ is a homeomorphism, it follows that $s_{k}=s_{k+1}$. Hence, $\{s_{k}\}$ is convergent. \\

    (II) If $r_{k}>0$ for any $k>k_{0}$, then for any $m\in \mathbb{N}$, (\ref{4.9}) assures that $\{\boldsymbol{c}_{k}\}$ is bounded. Moreover, (\ref{S-3'}) shows that $\{\boldsymbol{\gamma}_{k}\}$ is also bounded. Furthermore, by \cite[Lemma 2.2]{Qi2018} and (\ref{S-3}), we have that 
    $$
         \| \boldsymbol{\alpha}_{k}\|\leq \|\mathcal{A}_{K}^{2m}(\boldsymbol{c}_{k})^{2m-1}\|+\frac{1}{\beta}\|\boldsymbol{\gamma}_{k}-\boldsymbol{\gamma}_{k-1}\|\leq \|\mathcal{A}_{K}^{2m}\|_{F}\|\boldsymbol{c}_{k}\|^{2m-1}+\frac{1}{\beta}\left(\|\boldsymbol{\gamma}_{k}\|+\|\boldsymbol{\gamma}_{k-1}\|\right),
    $$
where $\|\cdot\|_{F}$ denotes the Frobenius norm of tensor (see \cite[Section 1.1]{Qi2018}). Thus, $\{(\boldsymbol{\alpha}_{k},\boldsymbol{c}_{k},\boldsymbol{\gamma}_{k})\}$ is bounded. Let $S$ be the set of subsequential limits of $\{(\boldsymbol{\alpha}_{k},\boldsymbol{c}_{k},\boldsymbol{\gamma}_{k})\}$. Then \cite[Theorem 3.6 and 3.7]{Rudin1976Principles} show that $S$ is nonempty compact, and 
    \begin{equation}\label{4.10}
       \lim_{k\rightarrow \infty} \text{\normalfont dist}((\boldsymbol{\alpha}_{k},\boldsymbol{c}_{k},\boldsymbol{\gamma}_{k}),S)=0.
    \end{equation}
where dist$(\cdot,\cdot)$ denotes Euclidean distance. By Assumption \ref{Assumption:4.1} (i), it follows that $\mathcal{L}_{\beta}$ is a KL function on $\mathbb{R}^{3N}$ and thus a KL function on $S$. Moreover, we verify that $\mathcal{L}_{\beta}$ is constant on $S$. 

    For any $(\boldsymbol{\alpha}_{*},\boldsymbol{c}_{*},\boldsymbol{\gamma}_{*})\in S$, there exists a subsequence $\{(\boldsymbol{\alpha}_{k_{j}},\boldsymbol{c}_{k_{j}},\boldsymbol{\gamma}_{k_{j}})\}$ that converges to $(\boldsymbol{\alpha}_{*},\boldsymbol{c}_{*},\boldsymbol{\gamma}_{*})$. Hence, the lower semi-continuity of $\mathcal{L}_{\beta}$ at $(\boldsymbol{\alpha}_{*},\boldsymbol{c}_{*},\boldsymbol{\gamma}_{*})$ and Corollary \ref{Corollary:4.4} show that
    \begin{equation}\label{4.11}
      \mathcal{L}_{\beta}(\boldsymbol{\alpha}_{*},\boldsymbol{c}_{*},\boldsymbol{\gamma}_{*})\leq \lim_{k\rightarrow \infty}\mathcal{L}_{\beta}(\boldsymbol{\alpha}_{k},\boldsymbol{c}_{k},\boldsymbol{\gamma}_{k})=\liminf_{j\rightarrow \infty} \mathcal{L}_{\beta}(\boldsymbol{\alpha}_{k_{j}},\boldsymbol{c}_{k_{j}},\boldsymbol{\gamma}_{k_{j}}).
    \end{equation}

    Conversely, since $\boldsymbol{\alpha}_{k_{j}+1}$ is the minimizer of $\mathcal{L}_{\beta}(\boldsymbol{\alpha},\boldsymbol{c}_{k_{j}} ,\boldsymbol{\gamma}_{k_{j}})$ and $\delta>0$, it shows that when $k_{j}>k_{0}$, 
     $$
        \mathcal{L}_{\beta}(\boldsymbol{\alpha}_{*},\boldsymbol{c}_{k_{j}},\boldsymbol{\gamma}_{k_{j}})\geq \mathcal{L}_{\beta}(\boldsymbol{\alpha}_{k_{j}+1},\boldsymbol{c}_{k_{j}},\boldsymbol{\gamma}_{k_{j}})\geq \mathcal{L}_{\beta}(\boldsymbol{\alpha}_{k_{j}+1},\boldsymbol{c}_{k_{j}+1},\boldsymbol{\gamma}_{k_{j}+1}).
     $$
From the continuity of $\mathcal{L}_{\beta}$ with respect to $\boldsymbol{c}$ and $\boldsymbol{\gamma}$, it holds
     $$
        \lim_{j\rightarrow \infty} \mathcal{L}_{\beta}(\boldsymbol{\alpha}_{*},\boldsymbol{c}_{k_{j}},\boldsymbol{\gamma}_{k_{j}})=\mathcal{L}_{\beta}(\boldsymbol{\alpha}_{*},\boldsymbol{c}_{*},\boldsymbol{\gamma}_{*}).
     $$
Combining with two inequalities above, we verify that
     $$
    \lim\limits_{j\rightarrow \infty}\mathcal{L}_{\beta}(\boldsymbol{\alpha}_{k_{j}+1},\boldsymbol{c}_{k_{j}+1},\boldsymbol{\gamma}_{k_{j}+1})=\lim_{k\rightarrow \infty}\mathcal{L}_{\beta}(\boldsymbol{\alpha}_{k},\boldsymbol{c}_{k},\boldsymbol{\gamma}_{k}).
     $$
From \cite[Theorem 3.19]{Rudin1976Principles}, the three relations above show that
    \begin{equation}\label{4.12}
       \mathcal{L}_{\beta}(\boldsymbol{\alpha}_{*},\boldsymbol{c}_{*},\boldsymbol{\gamma}_{*})\geq \lim_{k\rightarrow \infty}\mathcal{L}_{\beta}(\boldsymbol{\alpha}_{k},\boldsymbol{c}_{k},\boldsymbol{\gamma}_{k}).
    \end{equation}
Finally, (\ref{4.11}) and (\ref{4.12}) assure that
    \begin{equation*}
       \mathcal{L}_{\beta}(\boldsymbol{\alpha}_{*},\boldsymbol{c}_{*},\boldsymbol{\gamma}_{*})=\lim_{k\rightarrow \infty}\mathcal{L}_{\beta}(\boldsymbol{\alpha}_{k},\boldsymbol{c}_{k},\boldsymbol{\gamma}_{k}).
    \end{equation*}
Therefore, $\mathcal{L}_{\beta}$ is constant on $S$, that is, the image $\mathcal{L}_{\beta}(S)$ is a singleton and
    \begin{equation}\label{4.13}
        \lim_{k\rightarrow \infty}\mathcal{L}_{\beta}(\boldsymbol{\alpha}_{k},\boldsymbol{c}_{k},\boldsymbol{\gamma}_{k})=\mathcal{L}_{\beta}(S).
    \end{equation}
Hence, \cite[Lemma 3.6]{Bolte2014Proximal} assure that there exist $\varepsilon>0$, $\eta>0$ and a continuous concave function $\varphi:[0,\eta)\rightarrow (0,\infty)$ related to KL property such that 

    (i) $\varphi(0)=0$ and $\varphi$ is continuously differentiable on $(0,\eta)$ with positive derivatives;

    (ii) if $\text{dist}((\boldsymbol{\alpha},\boldsymbol{c},\boldsymbol{\gamma}),S)<\varepsilon$ and $\mathcal{L}_{\beta}(S)<\mathcal{L}_{\beta}(\boldsymbol{\alpha},\boldsymbol{c},\boldsymbol{\gamma})<\mathcal{L}_{\beta}(S)+\eta$, then
    $$
       \varphi^{'}(\mathcal{L}_{\beta}(\boldsymbol{\alpha},\boldsymbol{c},\boldsymbol{\gamma})-\mathcal{L}_{\beta}(S))\ \text{\normalfont{dist}}(0,\partial \mathcal{L}_{\beta}(\boldsymbol{\alpha},\boldsymbol{c},\boldsymbol{\gamma})) \geq 1.
    $$
From Corollary \ref{Corollary:4.4}, (\ref{4.10}) and (\ref{4.13}), it suffices to show that for $\varepsilon>0$ and $\eta>0$ above, there exists an integer $k_{1}>k_{0}$ such that for any $k>k_{1}$, we have
    \begin{equation}\label{4.14}
      \varphi^{'}(r_{k})\ \text{\normalfont{dist}}(0,\partial \mathcal{L}_{\beta}(\boldsymbol{\alpha}_{k},\boldsymbol{c}_{k},\boldsymbol{\gamma}_{k}))\geq 1.
    \end{equation}
From the concavity of $\varphi$, we get that
    $$
     \varphi^{'}(r_{k})(r_{k}-r_{k+1})\leq \varphi(r_{k})-\varphi(r_{k+1}).
    $$
Multiplying $\text{\normalfont{dist}}(0,\partial \mathcal{L}_{\beta}(\boldsymbol{\alpha}_{k},\boldsymbol{c}_{k},\boldsymbol{\gamma}_{k}))$ on both side and using (\ref{4.14}), we obtain that
    \begin{equation}\label{4.15}
      r_{k}-r_{k+1} \leq \text{\normalfont{dist}}(0,\partial \mathcal{L}_{\beta}(\boldsymbol{\alpha}_{k},\boldsymbol{c}_{k},\boldsymbol{\gamma}_{k})) (\varphi(r_{k})-\varphi(r_{k+1})).
    \end{equation}
By \cite[8.8 Exercise (c) and 10.5 proposition]{Rockafellar1998} and (\ref{S-3'}), it follows that
    \begin{equation}\label{4.16}
       \partial \mathcal{L}_{\beta}(\boldsymbol{\alpha}_{k},\boldsymbol{c}_{k},\boldsymbol{\gamma}_{k})=\partial_{\boldsymbol{\alpha}} \mathcal{L}_{\beta}(\boldsymbol{\alpha}_{k},\boldsymbol{c}_{k},\boldsymbol{\gamma}_{k})\times \nabla_{\boldsymbol{c}} \mathcal{L}_{\beta}(\boldsymbol{\alpha}_{k},\boldsymbol{c}_{k},\boldsymbol{\gamma}_{k})\times \nabla_{\boldsymbol{\gamma}} \mathcal{L}_{\beta}(\boldsymbol{\alpha}_{k},\boldsymbol{c}_{k},\boldsymbol{\gamma}_{k}),
    \end{equation}
where
    \begin{align*}
         &\partial_{\boldsymbol{\alpha}} \mathcal{L}_{\beta}(\boldsymbol{\alpha}_{k},\boldsymbol{c}_{k},\boldsymbol{\gamma}_{k})=\partial F(\boldsymbol{\alpha}_{k})+\boldsymbol{\gamma}_{k}+\beta\left(\boldsymbol{\alpha}_{k}-\mathcal{A}_{K}^{2m}(\boldsymbol{c}_{k})^{2m-1}\right), \notag \\
         &\nabla_{\boldsymbol{c}} \mathcal{L}_{\beta}(\boldsymbol{\alpha}_{k},\boldsymbol{c}_{k},\boldsymbol{\gamma}_{k})=-(2m-1)\mathcal{A}_{K}^{2m}(\boldsymbol{c}_{k})^{2m-2}\cdot \beta\left(\boldsymbol{\alpha}_{k}-\mathcal{A}_{K}^{2m}(\boldsymbol{c}_{k})^{2m-1}\right), \\
         &\nabla_{\boldsymbol{\gamma}} \mathcal{L}_{\beta}(\boldsymbol{\alpha}_{k},\boldsymbol{c}_{k},\boldsymbol{\gamma}_{k})=\boldsymbol{\alpha}_{k}-\mathcal{A}_{K}^{2m}(\boldsymbol{c}_{k})^{2m-1} \notag.
    \end{align*}
Invoking the optimality condition for (\ref{S-1'}), we have that
    \begin{equation}\label{4.17}
       -\beta \left(\boldsymbol{\alpha}_{k}-\mathcal{A}_{K}^{2m}(\boldsymbol{c}_{k-1})^{2m-1}+\frac{1}{\beta}\boldsymbol{\gamma}_{k-1}\right)\in \partial F(\boldsymbol{\alpha}_{k}), 
    \end{equation}
From (\ref{S-3}), (\ref{4.16}) and (\ref{4.17}), we obtain further that
    \begin{align*}
         \boldsymbol{\alpha}^{\#}_{k}&:=\beta(\mathcal{A}_{K}^{2m}(\boldsymbol{c}_{k-1})^{2m-1}-\mathcal{A}_{K}^{2m}(\boldsymbol{c}_{k})^{2m-1})-(\boldsymbol{\gamma}_{k-1}-\boldsymbol{\gamma}_{k})\in \partial_{\boldsymbol{\alpha}} \mathcal{L}_{\beta}(\boldsymbol{\alpha}_{k},\boldsymbol{c}_{k},\boldsymbol{\gamma}_{k}), \notag \\
         \boldsymbol{c}^{\#}_{k}&:=(2m-1)\mathcal{A}_{K}^{2m}(\boldsymbol{c}_{k})^{2m-2} \cdot(\boldsymbol{\gamma}_{k-1}-\boldsymbol{\gamma}_{k})=\nabla_{\boldsymbol{c}} \mathcal{L}_{\beta}(\boldsymbol{\alpha}_{k},\boldsymbol{c}_{k},\boldsymbol{\gamma}_{k}),  \\
         \boldsymbol{\gamma}^{\#}_{k}&:=-\frac{1}{\beta}(\boldsymbol{\gamma}_{k-1}-\boldsymbol{\gamma}_{k})= \nabla_{\boldsymbol{\gamma}} \mathcal{L}_{\beta}(\boldsymbol{\alpha}_{k},\boldsymbol{c}_{k},\boldsymbol{\gamma}_{k}). \notag
    \end{align*}
Hence, we have that $(\boldsymbol{\alpha}^{\#}_{k},\boldsymbol{c}^{\#}_{k},\boldsymbol{\gamma}^{\#}_{k})\in \partial \mathcal{L}_{\beta}(\boldsymbol{\alpha}_{k},\boldsymbol{c}_{k},\boldsymbol{\gamma}_{k})$. It means that
    \begin{equation}\label{4.18}
       \text{\normalfont{dist}}(0,\partial \mathcal{L}_{\beta}(\boldsymbol{\alpha}_{k},\boldsymbol{c}_{k},\boldsymbol{\gamma}_{k})) \leq \|(\boldsymbol{\alpha}^{\#}_{k},\boldsymbol{c}^{\#}_{k},\boldsymbol{\gamma}^{\#}_{k})\|\leq \| \boldsymbol{\alpha}^{\#}_{k} \|+\|\boldsymbol{c}^{\#}_{k}\|+\|\boldsymbol{\gamma}^{\#}_{k}\|.
    \end{equation}
Since $\{\boldsymbol{c}_{k}\}$ is bounded, there exists $\upsilon>0$ such that $\|\boldsymbol{c}_{k}\|\leq\upsilon$.
Let 
    $$
       \iota:=(2m-1)\|\mathcal{A}_{K}^{2m}\|_{F}\upsilon^{2m-2}.
    $$
Then \cite[Lemma 2.2]{Qi2018} shows that for any $\boldsymbol{c}\in \mathbb{R}^{N}$ such that $\|\boldsymbol{c}\|\leq \upsilon$, we have
    $$
       \|(2m-1)\mathcal{A}_{K}^{2m}\boldsymbol{c}^{2m-2}\| \leq (2m-1)\|\mathcal{A}_{K}^{2m}\|_{F}\|\boldsymbol{c}\|^{2m-2}\leq (2m-1)\|\mathcal{A}_{K}^{2m}\|_{F}\upsilon^{2m-2}=\iota.
    $$
Using the convex subset $\{\boldsymbol{c}:\|\boldsymbol{c}\|\leq \upsilon\}$ instead of $\mathbb{R}^{N}$ to reproof \cite[Theorem 5.12]{Beck2017}, we see that $\boldsymbol{c}\mapsto \mathcal{A}_{K}^{2m}\boldsymbol{c}^{2m-1}$ is $\iota$-smooth on $\{\boldsymbol{c}:\|\boldsymbol{c}\|\leq \upsilon\}$. Since for any $k\in \mathbb{N}$, $\|\boldsymbol{c}_{k-1}\|\leq \upsilon$ and $\|\boldsymbol{c}_{k}\|\leq \upsilon$, then 
    \begin{equation}\label{4.19}
       \|\mathcal{A}_{K}^{2m}(\boldsymbol{c}_{k-1})^{2m-1}-\mathcal{A}_{K}^{2m}(\boldsymbol{c}_{k})^{2m-1}\|\leq \iota\|\boldsymbol{c}_{k-1}-\boldsymbol{c}_{k}\|.
    \end{equation}
From (\ref{S-3'}), (\ref{4.4}) and (\ref{4.19}), we have that
    \begin{align*}
       \|\mathcal{A}_{K}^{2m}(\boldsymbol{c}_{k-1})^{2m-1}-\mathcal{A}_{K}^{2m}(\boldsymbol{c}_{k})^{2m-1}\| &\leq \iota w_{2}\|\Psi(s_{k-1})-\Psi(s_{k})\|_{\mathcal{B}_{K}^{2m}(X)} \\
       \|\boldsymbol{\gamma}_{k-1}-\boldsymbol{\gamma}_{k}\|=\frac{2m\lambda}{2m-1}\|\boldsymbol{c}_{k-1}-\boldsymbol{c}_{k}\|&\leq \frac{2m\lambda w_{2}}{2m-1} \|\Psi(s_{k-1})-\Psi(s_{k})\|_{\mathcal{B}_{K}^{2m}(X)}.
    \end{align*}
Inserting three inequalites above into $\boldsymbol{\alpha}^{\#}_{k}$, $\boldsymbol{c}^{\#}_{k}$ and $\boldsymbol{\gamma}^{\#}_{k}$, we verify that
      \begin{align}\label{4.20}
         \| \boldsymbol{\alpha}^{\#}_{k} \| &\leq \left(\beta\iota w_{2}+\frac{2m\lambda w_{2}}{2m-1}\right) \|\Psi(s_{k-1})-\Psi(s_{k})\|_{\mathcal{B}_{K}^{2m}(X)},\notag \\
         \|\boldsymbol{c}^{\#}_{k}\| &\leq \frac{2m\lambda w_{2}\iota}{2m-1} \|\Psi(s_{k-1})-\Psi(s_{k})\|_{\mathcal{B}_{K}^{2m}(X)}, \\
         \|\boldsymbol{\gamma}^{\#}_{k}\| &\leq \frac{2m\lambda w_{2}}{(2m-1)\beta}  \|\Psi(s_{k-1})-\Psi(s_{k})\|_{\mathcal{B}_{K}^{2m}(X)} \notag.
      \end{align}
From (\ref{4.18}) and (\ref{4.20}), it follows that there exists $\zeta_{2}>0$ such that
      \begin{equation}\label{4.21}
         \text{\normalfont{dist}}(0,\partial \mathcal{L}_{\beta}(\boldsymbol{\alpha}_{k},\boldsymbol{c}_{k},\boldsymbol{\gamma}_{k})) \leq \zeta_{2} \|\Psi(s_{k-1})-\Psi(s_{k})\|_{\mathcal{B}_{K}^{2m}(X)}.
      \end{equation}    
Since $\zeta_{1}, \zeta_{2}>0$, whenever $k>k_{1}+1$, Lemma \ref{Lemma:4.3}, (\ref{4.15}) and (\ref{4.21}) assure that
    $$
     \zeta_{1}\|\Psi(s_{k})-\Psi(s_{k+1})\|_{\mathcal{B}_{K}^{2m}(X)}^{2}\leq \zeta_{2}\|\Psi(s_{k-1})-\Psi(s_{k})\|_{\mathcal{B}_{K}^{2m}(X)}\left(\varphi(r_{k})-\varphi(r_{k+1})\right).
    $$
By the mean inequality and rearranging term, we have that
    $$
     \|\Psi(s_{k})-\Psi(s_{k+1})\|_{\mathcal{B}_{K}^{2m}(X)} \leq \frac{\|\Psi(s_{k-1})-\Psi(s_{k})\|_{\mathcal{B}_{K}^{2m}(X)}+\dfrac{\zeta_{2}}{\zeta_{1}} (\varphi(r_{k})-\varphi(r_{k+1}))}{2}.
    $$
Summing up the above relation from $k=k_{1}+1$ to $\infty$ and rearranging terms, since $\varphi(r_{k_{1}+1})>0$, we see that
    $$
       \sum_{k=k_{1}+1}^{\infty} \|\Psi(s_{k})-\Psi(s_{k+1})\|_{\mathcal{B}_{K}^{2m}(X)} \leq \|\Psi(s_{k_{1}})-\Psi(s_{k_{1}+1})\|_{\mathcal{B}_{K}^{2m}(X)}+ \frac{\zeta_{2}}{\zeta_{1}}\varphi(r_{k_{1}+1}).
    $$
Thus, it shows that
    \begin{equation}\label{4.22}
       \sum\limits_{k\in \mathbb{N}} \|\Psi(s_{k})-\Psi(s_{k+1})\|_{\mathcal{B}_{K}^{2m}(X)}<\infty,
    \end{equation}
which means that $\{\Psi(s_{k})\}$ is a Cauchy sequence by triangle inequality. Since $\mathcal{B}_{K}^{2m}(X)$ is a Banach space which is a complete metric space, the convergence of $\{\Psi(s_{k})\}$ follows immediately from this. Since $\Psi$ is a homeomorphism, then the convergence of $\{\Psi(s_{k})\}$ implies the convergence of $\{s_{k}\}$. \\

    Combining (I) with (II), we conclude that $\{s_{k}\}$ is convergent. We denote
    $$
      s_{*}:=\lim_{k\to \infty} s_{k}.
    $$
Since $\{s_{k}\}\subseteq \Delta_{K}^{\frac{2m}{2m-1}}(X)$ and $\Delta_{K}^{\frac{2m}{2m-1}}(X)$ is closed, $s^{*} \in \Delta_{K}^{\frac{2m}{2m-1}}(X)$. Hence, there exists $\boldsymbol{c}_{*}\in \mathbb{R}^{N}$ such that $s^{*}$ has the representation
    $$
       s_{*}=\sum_{n\in \mathbb{N}}\left(\boldsymbol{\Phi}_{n}^{T}\boldsymbol{c}_{*}\right)^{2m-1}\phi_{n}.
    $$ 
Since $s_{k}\to s_{*}$ when $k\to \infty$, we observe that 
    $$
       \lim_{k\to \infty} \Psi(s_{k})=\Psi(s^{*})\ \  \text{and}\ \lim_{k\to \infty} \boldsymbol{c}_{k}=\boldsymbol{c}_{*}.
    $$
In particular, Assumption \ref{Assumption:4.1} (ii) and (iii) show that $\boldsymbol{c}_{*}\neq \boldsymbol{0}$ which means that $s_{*}\neq 0$. Next we show that $s_{*}$ is a stationary point of Optimization (\ref{2.1}). Let $\boldsymbol{\delta}:=(\delta_{\boldsymbol{x}_{1}},\delta_{\boldsymbol{x}_{2}},...,\delta_{\boldsymbol{x}_{N}})^{T}$, and $\delta_{\boldsymbol{x}_{i}}:\mathcal{B}_{K}^{\frac{2m}{2m-1}}(X)\to \mathbb{R}$ is the Dirac functional, that is, 
    $$
        \delta_{\boldsymbol{x}_{i}}(f):=f(\boldsymbol{x}_{i}),\ i=1,2,...,N.
    $$ 
Also, we denote the objective function of Optimization (\ref{2.1}) as $\mathcal{T}:\mathcal{B}_{K}^{\frac{2m}{2m-1}}(X)\to \mathbb{R}$. Thus, we see that
    $$
        \mathcal{T}(f)=\frac{1}{N} \sum_{i=1}^{N}L(\boldsymbol{x}_{i},y_{i},\delta_{\boldsymbol{x}_{i}}(f))+\lambda \|f\|_{\mathcal{B}_{K}^{\frac{2m}{2m-1}}(X)}^{\frac{2m}{2m-1}}.
    $$
Since $\delta_{\boldsymbol{x}_{i}}(f)=f(\boldsymbol{x}_{i})=\sum\limits_{n\in \mathbb{N}}a_{n}\phi_{n}(\boldsymbol{x}_{i}),\ i=1,2,...,N$, by derivative rule, we see that
    \begin{equation}\label{4.23}
       \nabla \boldsymbol{\delta}(s_{*})=\left(\sum_{n\in \mathbb{N}}\phi_{n}(\boldsymbol{x}_{1})\phi_{n},\sum_{n\in \mathbb{N}}\phi_{n}(\boldsymbol{x}_{2})\phi_{n},...,\sum_{n\in \mathbb{N}}\phi_{n}(\boldsymbol{x}_{N})\phi_{n}\right)^{T}.
    \end{equation}
Moreover, from (\ref{S-3}) and (\ref{S-3'}), it follows that the sequence $\{(\boldsymbol{\alpha}_{k},\boldsymbol{c}_{k},\boldsymbol{\gamma}_{k})\}$ also converges to $(\boldsymbol{\alpha}_{*},\boldsymbol{c}_{*},\boldsymbol{\gamma}_{*})\in S$, and
    $$
       \boldsymbol{\alpha}_{*}=\mathcal{A}_{K}^{2m}(\boldsymbol{c}_{*})^{2m-1},\ \boldsymbol{\gamma}_{*}=\frac{2m\lambda}{2m-1} \boldsymbol{c}_{*},
    $$
which ensures that
    \begin{equation}\label{4.24}
        \mathcal{L}_{\beta}(\boldsymbol{\alpha}_{*},\boldsymbol{c}_{*},\boldsymbol{\gamma}_{*})=F(\boldsymbol{\alpha}_{*})+G(\boldsymbol{c}_{*})=\mathcal{L}_{\beta}(S).
    \end{equation}
By (\ref{4.13}), (\ref{4.24}) and the continuity of $G$, we observe that
    $$
       \lim_{k\rightarrow \infty}F(\boldsymbol{\alpha}_{k})=F(\boldsymbol{\alpha}_{*})+G(\boldsymbol{c}_{*})-G(\boldsymbol{c}_{*})-0-0=F(\boldsymbol{\alpha}_{*}).
    $$
    In the view of (\ref{4.17}) and (\ref{S-3'}), by \cite[proposition 8.7]{Rockafellar1998} and passing to the limit along the sequence $\{(\boldsymbol{\alpha}_{k},\boldsymbol{c}_{k},\boldsymbol{\gamma}_{k})\}$, it follows that
    \begin{equation}\label{4.25}
       -\frac{2m\lambda}{2m-1} \boldsymbol{c}_{*}=-\boldsymbol{\gamma}_{*}\in \partial F(\boldsymbol{\alpha}_{*})=\partial F(\mathcal{A}_{K}^{2m}(\boldsymbol{c}_{*})^{2m-1}).
    \end{equation}
By (\ref{4.23}) and (\ref{4.25}), the chain rule shows that 
    $$
       (\nabla \boldsymbol{\delta}(s_{*}))^{T}(-\frac{2m\lambda}{2m-1} \boldsymbol{c}_{*})=-\frac{2m\lambda}{2m-1}\Psi(s_{*})\in \partial \left(\frac{1}{N} \sum_{i=1}^{N}L(\boldsymbol{x}_{i},y_{i},\delta_{\boldsymbol{x}_{i}}(\cdot))\right)(s_{*}).
    $$
On the other hand, since $s_{*}\neq 0$, (\ref{2.2}) assures that
    $$
       \nabla(\lambda \|\cdot\|_{\mathcal{B}_{K}^{\frac{2m}{2m-1}}(X)}^{\frac{2m}{2m-1}})(s_{*})=\frac{2m\lambda}{2m-1} \|s^{*}\|_{\mathcal{B}_{K}^{\frac{2m}{2m-1}}(X)}^{\frac{1}{2m-1}} \nabla(\|\cdot\|_{\mathcal{B}_{K}^{\frac{2m}{2m-1}}(X)}^{\frac{2m}{2m-1}})(s_{*})=\frac{2m\lambda}{2m-1}\Psi(s_{*}).
    $$
Combining with two relations above, \cite[Definition 1.77 and Proposition 1.107]{Mordukhovich2006} assure that
    $$
       0\in \partial \mathcal{T}(s_{*})=\partial \left(\frac{1}{N} \sum_{i=1}^{N}L(\boldsymbol{x}_{i},y_{i},\delta_{\boldsymbol{x}_{i}}(\cdot))\right)(s_{*})+\nabla(\lambda \|\cdot\|_{\mathcal{B}^{\frac{2m}{2m-1}}(X)}^{\frac{2m}{2m-1}})(s_{*}),
    $$
that is, $s_{*}$ is a stationary point of Optimization (\ref{2.1}). This proof is completed.
   \end{proof}

    \begin{remark}
      Similariy, we can deduce that $\{\boldsymbol{c}_{k}\}$ converges to a stationary point of Optimization (\ref{2.5}). We give an example of ADMM ((\ref{S-1}), (\ref{S-2}) and (\ref{S-3})) for the sum of a lower semi-continuous function and a convex function with nonlinear constraint.
   \end{remark}

    \section{Numerical Examples}
    \label{sec:5}
    In this section, we test Algorithm \ref{alg:ADMM} by the synthetic data and the real data for binary classification. We choose some training data and testing data, loss functions and kernels to test Algorithm \ref{alg:ADMM}. For simplicity, let $L_{1}$, $L_{2}$, $L_{3}$ and $L_{4}$ be four loss functions used in our experiments, that is,
    $$
      L_{1}(\boldsymbol{x},y,t)= \begin{cases} -yt+1, & yt-1<0 \\ 0, & yt-1\geq 0 \end{cases},\
       \ L_{2}(\boldsymbol{x},y,t)= \begin{cases} (-yt+1)^{2}, & yt-1<0  \\ 0, & yt-1 \geq 0 \end{cases},
    $$
and
    $$
      L_{3}(\boldsymbol{x},y,t)=\begin{cases} ln(2-yt), & yt-1<0 \\ 0, & yt-1 \geq 0  \end{cases},\ L_{4}(\boldsymbol{x},y,t)= \begin{cases} -yt+2, & yt-1<-1  \\ -2yt+2, & -1\leq yt-1<0  \\ 0, & yt-1 \geq 0 \end{cases}.
    $$
We see that $L_{1}$ is convex Hinge loss, $L_{2}$ is convex squared Hinge loss, $L_{3}$ is a nonconvex piecewise logarithmic loss function and $L_{4}$ is a nonconvex linear piecewise loss function. All of these loss functions satisfy Assumption \ref{Assumption:4.1} (i), (ii) and (iii). 
    \begin{figure}[h]
    \centering
    \includegraphics[width=3.6 in]{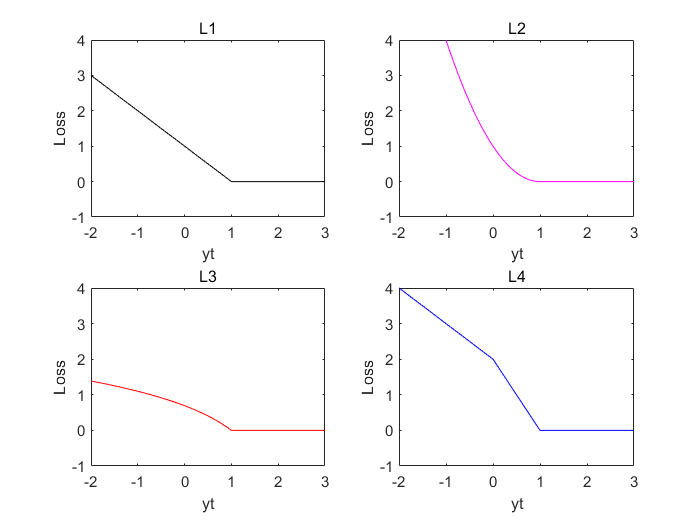}
    \caption{The support vector loss function $L_{1}$, $L_{2}$, $L_{3}$ and $L_{4}$. All are shown as a function of $yt$ rather than $t$, because of the symmetry between the $y=+1$ and $y=-1$ case.}
    \end{figure} 

    Let $K_{1}$ be the Gaussian kernel, that is, 
    $$
       K_{1}(\boldsymbol{x},\boldsymbol{x}')=\exp(-\sigma^{2} \| \boldsymbol{x}-\boldsymbol{x}'\|^{2})=\sum_{\boldsymbol{n}\in (\mathbb{N}_{0})^{d}}\phi_{\boldsymbol{n}}(\boldsymbol{x})\phi_{\boldsymbol{n}}(\boldsymbol{x}'),\ \sigma>0,\ \forall \boldsymbol{x}, \boldsymbol{x}'\in X,
    $$  
where $\phi_{\boldsymbol{n}}(\boldsymbol{x})=\prod_{j=1}^{d}\left(\frac{2^{n_{j}}}{n_{j}!}\right)^{\frac{1}{2}}(\sigma x_{j})^{n_{j}}\exp^{-\sigma^{2}(x_{j})^{2}},\ \forall \boldsymbol{x}=(x_{1},x_{2},...,x_{d})^{T}\in X$.
Also, let $K_{2}$ be the min kernel, that is, for any $\boldsymbol{x}, \boldsymbol{x}'\in X$,
    $$
      K_{2}(\boldsymbol{x},\boldsymbol{x}')=\prod_{j=1}^{d} \left(\min \{x_{j},x_{j}'\}-x_{j}x_{j}'\right)=\sum_{\boldsymbol{n}\in \mathbb{N}^{d}} \phi_{\boldsymbol{n}}(\boldsymbol{x})\phi_{\boldsymbol{n}}(\boldsymbol{x}'), 
    $$
where $\phi_{\boldsymbol{n}}(\boldsymbol{x})=\prod_{j=1}^{d} \frac{\sqrt{2}}{n_{j}\pi}\sin(n_{j}\pi x_{j}),\ \forall \boldsymbol{x}=(x_{1},x_{2},...,x_{d})^{T}\in X$. For any $m\in \mathbb{N}$, we can construct $\mathcal{B}_{K_{1}}^{\frac{2m}{2m-1}}(X)$ and $\mathcal{B}_{K_{2}}^{\frac{2m}{2m-1}}(X)$.

    In each experiment, we select the training data, a loss function and a RKBS above and some suitable parameters, and use Algorithm \ref{alg:ADMM} for training. Since $\{s_{k}\}$ is globally convergent to a stationary point of Optimization (\ref{2.1}), we solve Optimization (\ref{2.1}) repeatedly by selecting some initial values randomly and choosing the minimizer of these outputs as the approximate solution $s_{D}^{\frac{2m}{2m-1}}:X\to \mathbb{R}$ to build the SVM
    $$
        \mathcal{R}s_{D}^{\frac{2m}{2m-1}}(\boldsymbol{x})=\begin{cases} +1, & s_{D}^{\frac{2m}{2m-1}}(\boldsymbol{x}) \geq0,\\
                                 -1, & s_{D}^{\frac{2m}{2m-1}}(\boldsymbol{x})<0.
                   \end{cases}   
    $$
to make the prediction on testing data. Recall that $\mathcal{A}_{K}^{2m}$ is the sum of infinite terms and is absolutely uniformly convergent, we truncate $\mathcal{A}_{K}^{2m}$ by $M$ terms and use $(\mathcal{A}_{K}^{2m})_{M}$ instead of $\mathcal{A}_{K}^{2m}$ to deal with the infinite sum and reduce the computation, that is, 
    $$
      \mathcal{A}_{K}^{2m}\approx (\mathcal{A}_{K}^{2m})_{M}=\sum\limits_{n\in \mathbb{N}_{M}}(\boldsymbol{\Phi}_{n})^{\otimes 2m}.
    $$
It is clear that $\lim\limits_{M\to \infty}(\mathcal{A}_{K}^{2m})_{M}=\mathcal{A}_{K}^{2m}$. If $M\geq N(N+1)/2$ and $(\mathcal{A}_{K}^{2m})_{M}$ satisfies Assumption \ref{Assumption:4.1} (iv), then we make sure the convergence of Algorithm \ref{alg:ADMM} with the loss functions and kernels above. Next, we discuss how to solve Optimization (\ref{S-1'}) by breaking them in $N$ Optimizations in $\mathbb{R}$. To illustrate above, we give a simple example. For the loss function $L_{1}$, by simple algebra, we have that if $y_{i}=+1$, then
    $$
       (\boldsymbol{\alpha}_{k+1})_{i}=
       \begin{cases}
           e_{i}+\frac{1}{\beta N} & e_{i}<1-\frac{1}{\beta N}, \\
           1 & 1-\frac{1}{\beta N}\leq e_{i}<1, \\
           e_{i} & e_{i}\geq 1,
       \end{cases}\ \ i=1,2,...,N,
    $$
where $e_{i}=(\mathcal{A}_{K}^{2m}(\boldsymbol{c}_{k})^{2m-1})_{i}-\frac{1}{\beta}(\boldsymbol{\gamma}_{k})_{i}$. On the other hand, if $y_{i}=-1$, then
    $$
       (\boldsymbol{\alpha}_{k+1})_{i}=
       \begin{cases}
           e_{i} & e_{i}<-1, \\
           -1 & -1\leq e_{i}<-1+\frac{1}{\beta N}, \\
           e_{i}-\frac{1}{\beta N} & e_{i}\geq -1+\frac{1}{\beta N}.
       \end{cases}
    $$
    
    Next we introduce our test results on synthetic data and real data. 

    \subsection{Examples on Synthetic Data}
    In this subsection, we introduce our test results on the synthetic data. First, we use a training set $D_{11}$ with 25 points and a testing set $D_{12}$ with 2601 points to show the effectiveness of Algorithm \ref{alg:ADMM}. 

    \vspace{-0.5cm}
    \begin{figure}[H]
	\centering
	\subfloat[Training Set $D_{11}$]{
		\includegraphics[width=1.8 in]{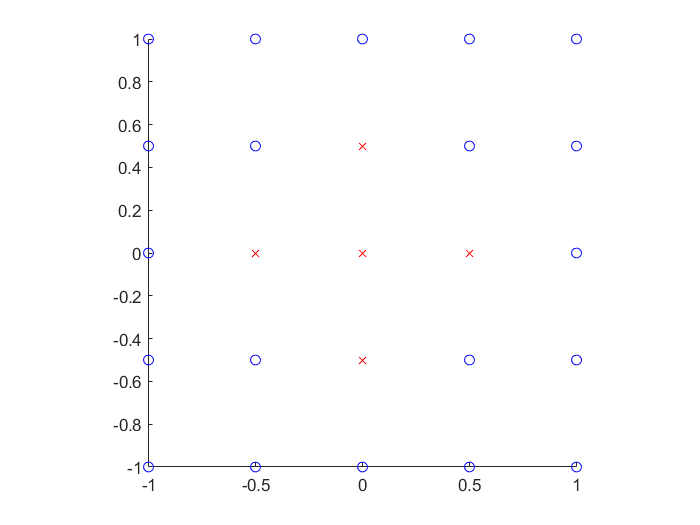}
		
    }
	\subfloat[Testing Set $D_{12}$]{
		\includegraphics[width=1.8 in]{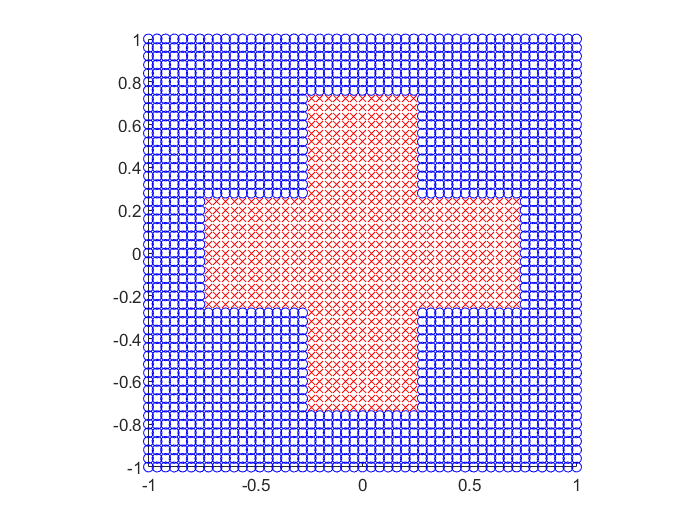}
		
	}
	\caption{The binary classification $X_{1}=[-1,1]\times [-1,1]$ and $Y=\{+1,-1\}$: The classes are coded as a binary variable (blue=$+1$ and red=$-1$). The left panel represents the training data $D_{11}$ and the right panel represents the testing data $D_{12}$.}
    \end{figure}
    \vspace{-0.5cm}

    First, we show the convergence of Algorithm \ref{alg:ADMM}. The selection of relevant parameters in this numerical experiments is given below which satisfies Assumption \ref{Assumption:4.1}.

     \begin{itemize}
      \setlength{\itemsep}{0pt}
      \setlength{\parskip}{0pt}
       \item Gaussian kernel $K_{1}$, where $\sigma=1$.
       \item The loss functions $L_{1}$, $L_{2}$, $L_{3}$ and $L_{4}$.
       \item The RKBS $\mathcal{B}_{K_{1}}^{2}(X_{1})$, $\mathcal{B}_{K_{1}}^{\frac{4}{3}}(X_{1})$ and $\mathcal{B}_{K_{1}}^{\frac{6}{5}}(X_{1})$.
       \item $N=25$, $\lambda=0.04$, $\beta=0.1$, $\mu=10^{-6}$ and $\varepsilon_{0}=10^{-12}$.
       \item we choose 20 initial values randomly in $[-1,1]^{N}$.
     \end{itemize}   

    By Theorem \ref{Theorem:4.2}, we can show the convergence of the sequence $\{s_{k}\}$ by 
    $$
       \sum\limits_{k\in \mathbb{N}}\|\Psi(s_{k+1})-\Psi(s_{k})\|_{\mathcal{B}_{K_{1}}^{2m}(X_{1})}=\sum_{k\in \mathbb{N}}\left(\sum_{n\in \mathbb{N}} \left|\boldsymbol{\Phi}_{n}^{T}(\boldsymbol{c}_{k+1}-\boldsymbol{c}_{k})\right|^{2m} \right)^{\frac{1}{2m}}. 
    $$
We use $(\mathcal{A}_{K}^{2m})_{M}$ instead of $\mathcal{A}_{K}^{2m}$ and we have that
    $$
      \sum_{k\in \mathbb{N}}\|\Psi(s_{k+1})-\Psi(s_{k})\|_{\mathcal{B}_{K_{1}}^{2m}(X_{1})}\approx \sum_{k\in \mathbb{N}}\left(\sum_{n\in \mathbb{N}_{M}} \left|\boldsymbol{\Phi}_{n}^{T}(\boldsymbol{c}_{k+1}-\boldsymbol{c}_{k})\right|^{2m} \right)^{\frac{1}{2m}}.
    $$

    We now conduct experiments to verify convergence of Algorithm \ref{alg:ADMM} with nonconvex loss function $L_{3}$. 

    \vspace{-0.5cm}
    \begin{figure}[H]
	\centering
	\subfloat[$m=1$]{
		\includegraphics[width=1.6 in]{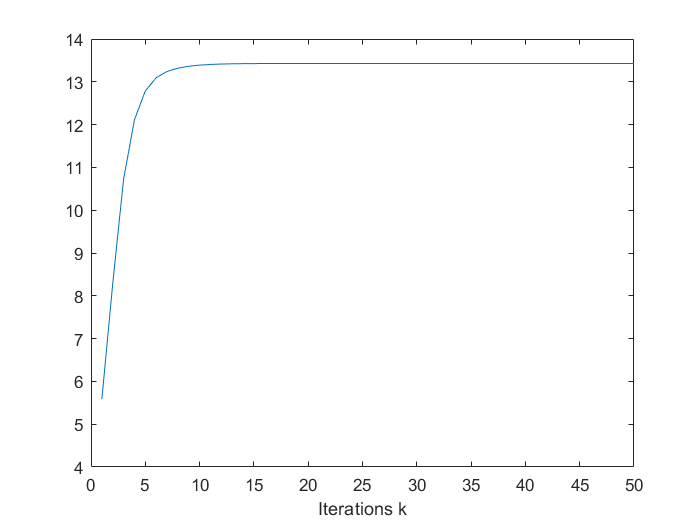}
		
    }
	\subfloat[$m=2$]{
		\includegraphics[width=1.6 in]{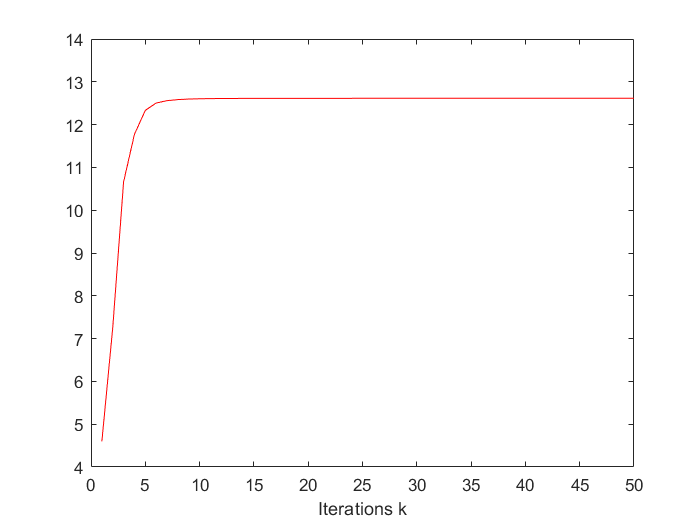}
		
	}
    \subfloat[$m=3$]{
		\includegraphics[width=1.6 in]{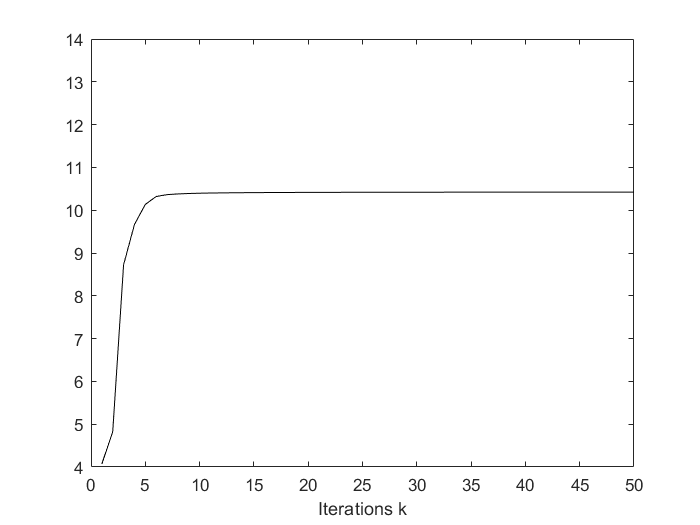}
		
	}
    \caption{The convergence of Algorithm \ref{alg:ADMM} in different RKBSs with the loss function $L_{3}$. The horizontal axis represents iterations and the vertical axis represents $\sum\limits_{k\in \mathbb{N}}\|\Psi(s_{k+1})-\Psi(s_{k})\|_{\mathcal{B}_{K_{1}}^{2m}(X_{1})}$. }
    \label{Figure:3}
    \end{figure}
    
     As shown in Figure \ref{Figure:3}, for the training data $D_{11}$, loss function $L_{3}$ and other parameters above, Algorithm \ref{alg:ADMM} converges in less than 50 iterations. These numerical results shows that Algorithm \ref{alg:ADMM} is efficient and stable. Next, we use other loss functions mentioned above to test Algorithm \ref{alg:ADMM} and reveal the advantage of the SVM in general RKBS with nonconvex loss function through comparing the performance in different RKBSs with different lower semi-continuous loss functions. Here are the results of these experiments:

    \begin{table}[H]
    \caption{Different Testing Accuracy on Testing Set $D_{12}$.}
    \centering
   {\begin{tabular}{c|c|c|c|c} \hline
         \diagbox[dir=NW]{RKBS}{Loss} & $L_{1}$ & $L_{2}$ & $L_{3}$ & $L_{4}$ \\ \hline
         $\mathcal{B}_{K_{1}}^{2}(X_{1})$ & 90.3\% & 92.0\% & 89.1\% & 91.4\% \\ \hline
         $\mathcal{B}_{K_{1}}^{\frac{4}{3}}(X_{1})$ & 90.3\% & 91.4\% & 91.2\% & 92.5\% \\ \hline
         $\mathcal{B}_{K_{1}}^{\frac{6}{5}}(X_{1})$ & 91.2\% & 91.4\% & 85.2\% & 93.2\% \\ \hline
    \end{tabular}}
    \label{table:1}
    \end{table}

    From Table \ref{table:1}, it shows that the SVM in RKBS with lower semi-continuous loss function by Algorithm \ref{alg:ADMM} is feasible in terms of accuracy. Moreover, it is easy to see that for this training data $D_{11}$ and testing data $D_{12}$, the SVM in $\mathcal{B}_{K_{1}}^{\frac{6}{5}}(X_{1})$ with the nonconvex loss function $L_{4}$ performs better than other case shown in Table \ref{table:1}. However, when $m$ is larger, the performance of corresponding SVM is not necessarily better, such as the performance of $L_{3}$ above. Next, we introduce the numerical experiment result on another dataset.

    We sample from $\Omega_{1}=[0.4,1]\times[0.4,1]$ labeled by $+1$ and $\Omega_{2}=[0,0.6]\times[0,0.6]$ labeled by $-1$ randomly to obtain training set $D_{21}$ and testing set $D_{22}$. The data labeled by $+1$ are equal to the data labeled by $-1$ in $D_{21}$ or $D_{22}$. Here is an example of sampling. In the following figures, two subdatasets are colored in blue and red.  

    \vspace{-0.75cm}
    \begin{figure}[H]
	\centering
	\subfloat[Training Set $D_{21}$]{
		\includegraphics[width=1.8 in]{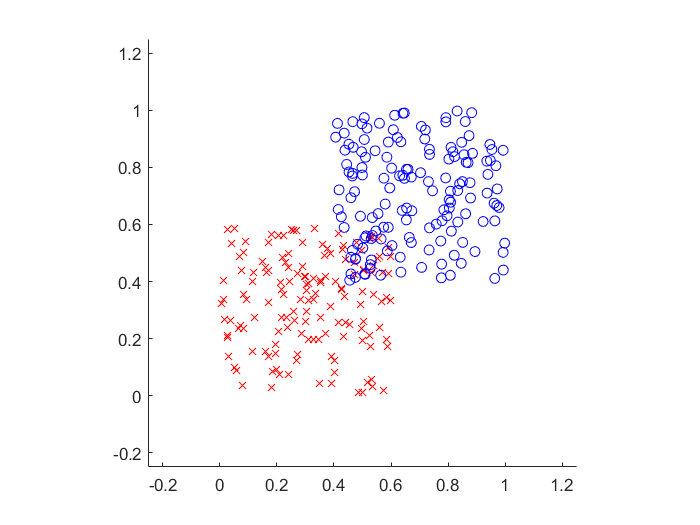}
		
    }
	\subfloat[Testing Set $D_{22}$]{
		\includegraphics[width=1.8 in]{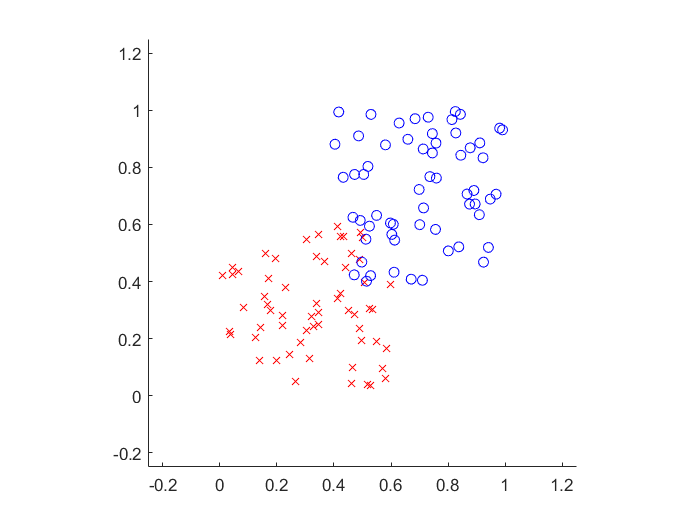}
		
	}
	\caption{An example of sampling where $X_{2}=[0,1]\times [0,1]$ and $Y=\{+1,-1\}$. The classes are coded as a binary variable (blue=$+1$ and red=$-1$). The left panel represents the training data $D_{21}$ and the right panel represents the testing data $D_{22}$.}
    \end{figure}
    \vspace{-0.5cm}

    Next, we use a training set $D_{21}$ with 300 points and a testing set $D_{22}$ with 120 points to show the effectiveness of Algorithm \ref{alg:ADMM}. Here are some parameters and results of these experiments:

      \begin{itemize}
      \setlength{\itemsep}{0pt}
      \setlength{\parskip}{0pt}
       \item Min kernel $K_{2}$.
       \item The loss functions $L_{1}$, $L_{2}$, $L_{3}$ and $L_{4}$.
       \item The RKBS $\mathcal{B}_{K_{2}}^{2}(X_{2})$, $\mathcal{B}_{K_{2}}^{\frac{4}{3}}(X_{2})$ and $\mathcal{B}_{K_{2}}^{\frac{6}{5}}(X_{2})$.
       \item $N=300$, $\lambda=0.01$, $\beta=1$ and $\varepsilon_{0}=10^{-12}$.
       \item Choose $20$ initial values randomly in $[0,1]^{N}$.
      \end{itemize}  

    \begin{table}[H]
    \caption{Different Testing Accuracy on Testing Set $D_{22}$.}
    \centering
   {\begin{tabular}{c|c|c|c|c} \hline
         \diagbox[dir=NW]{RKBS}{Loss} & $L_{1}$ & $L_{2}$ & $L_{3}$ & $L_{4}$ \\ \hline
         $\mathcal{B}_{K_{2}}^{2}(X_{2})$ & 90.0\% & 90.0\% & 90.0\% & 89.2\% \\ \hline
         $\mathcal{B}_{K_{2}}^{\frac{4}{3}}(X_{2})$ & 89.2\% & 90.0\% & 89.2\% & 90.8\% \\ \hline
         $\mathcal{B}_{K_{2}}^{\frac{6}{5}}(X_{2})$ & 90.0\% & 90.0\% & 89.2\% & 90.0\% \\ \hline
    \end{tabular}}
    \label{table:2}
    \end{table}

    From Table \ref{table:2}, we check that the SVM in $\mathcal{B}_{K_{2}}^{\frac{4}{3}}(X_{2})$ with nonconvex loss function $L_{4}$  performs better than others. It shows that in some cases the SVM in RKBS is more suitable than the classical SVM in RKHS. Next we introduce our experiments on real data.

     \subsection{Examples on UCI Machine Learning Repository}
    We choose the banknote authentication dataset in UCI Machine Learning Repository to test Algorithm \ref{alg:ADMM}. The data were extracted from images that were taken from genuine and forged banknote-like specimens. Wavelet Transform tool were used to extract features from images. There are 4 input variables about them, which are variance, skewness, curtosis of wavelet transformed image and entropy of image. Using principal component analysis, we select $3$ main input variables (variance, skewness, curtosis of wavelet transformed image) and denote the sample space $X_{3}=[-20,20]^{3}$ and the label space $Y=\{+1,-1\}$. Next, we choose $200$ images randomly as training set $D_{31}$, and a half of them are labeled by $+1$ and the others are labeled by $-1$. Also, we choose $100$ images randomly as testing set $D_{32}$, and a half of them are labeled by $+1$ and the others are labeled by $-1$. We introduce some parameters of these experiments:
    \begin{itemize}
    \setlength{\itemsep}{0pt}
    \setlength{\parskip}{0pt}
      \item The kernels $K_{1}$, where $\sigma=0.5$.
      \item The loss functions $L_{1}$, $L_{2}$, $L_{3}$ and $L_{4}$.
      \item The RKBS $\mathcal{B}_{K_{1}}^{2}(X_{3})$, $\mathcal{B}_{K_{1}}^{\frac{4}{3}}(X_{3})$ and $\mathcal{B}_{K_{1}}^{\frac{6}{5}}(X_{3})$.
      \item $N=200$, $\lambda=0.01$, $\beta=0.01$ and $\varepsilon_{0}=10^{-12}$.
      \item Choose $20$ initial values randomly in $[-1,1]^{N}$.
    \end{itemize}

    In each experiment, we will choose a loss function and an RKBS. Then we have the following results.

    \begin{table}[H]
    \caption{Different Testing Accuracy of on Testing Set $D_{32}$.}
    \centering
   {\begin{tabular}{c|c|c|c|c} \hline
         \diagbox[dir=NW]{RKBS}{Loss} & $L_{1}$ & $L_{2}$ & $L_{3}$ & $L_{4}$ \\ \hline
         $\mathcal{B}_{K_{1}}^{2}(X_{3})$ & 87.0\% & 85.0\% & 86.0\% & 79.0\% \\ \hline
         $\mathcal{B}_{K_{1}}^{\frac{4}{3}}(X_{3})$ & 86.0\% & 82.0\% & 88.0\% & 86.0\% \\ \hline
         $\mathcal{B}_{K_{1}}^{\frac{4}{3}}(X_{3})$ & 86.0\% & 82.0\% & 86.0\% & 86.0\% \\ \hline
    \end{tabular}}
    \label{table:3}
    \end{table}

    From Table \ref{table:3}, we check that the SVM in $\mathcal{B}_{K_{1}}^{\frac{4}{3}}(X_{3})$ with nonconvex loss function $L_{3}$ performs better than others in these experiments. It shows that in some cases the nonconvex loss function and RKBS are more suitable than the convex loss function and RKHS, which is our motivation of this paper.

    In Section \ref{sec:5}, we demonstrate the effectiveness of solving the SVM in RKBS with lower semi-continuous loss function by splitting method based on ADMM. In addition, we give some examples to show that in some cases, the SVM in RKBS with lower semi-continuous loss function is better than the SVM in RKHS with convex loss function. Therefore, we can consider not only RKHS and convex loss function, but also RKBS and nonconvex loss function.

    \section{Final Remarks} 
    \label{sec:6} 
    In the paper \cite{Lin2022}, the second author and the third author propose several numerical tricks in RKBS and discuss the homotopy method for the multikernel-based approximation method. As a continuation of the program, in this paper, we discuss the splitting method based on ADMM for the SVM in $\mathcal{B}_{K}^{\frac{2m}{2m-1}}(X)$ with lower semi-continuous loss function. Since $\mathcal{B}_{K}^{p}(X)\ (1\leq p<\infty)$ are also RKBSs, from \cite[Chapter 5]{Xu2019} and \cite[Corollary 4.1]{Huang2019}, the SVM in $\mathcal{B}_{K}^{p}(X)\ (1\leq p<\infty)$ have a minimizer. Since $\mathcal{B}_{K}^{p}(X)\ (p>1)$ has similar property with $\mathcal{B}_{K}^{\frac{2m}{2m-1}}(X)$, we try to use similar line of arguments therein to deal with. Although $\mathcal{B}_{K}^{1}(X)$ is not reflexive, strictly convex or smooth, we can use the SVM in $\mathcal{B}_{K}^{p}(X)\ (1<p<\infty)$ to approximate the SVM in $\mathcal{B}_{K}^{1}(X)$. Next we will study how to solve the SVM in $\mathcal{B}_{K}^{p}(X)\ (1\leq p<\infty)$ with lower semi-continuous loss function.

\hbox to14cm{\hrulefill}\par
   \  \\

   {\small Mingyu Mo} \par
   School of Mathematical Sciences \par
   South China Normal University \par
   Guangzhou, 510631, Guangdong, PR China \par
   Email-address: mmymaths@qq.com. \\
   \ \\

   {\small Yimin Wei} \par
   School of Mathematical Sciences \par
   Fudan University \par
   Shanghai, PR China \par
   Email-address: ymwei@fudan.edu.cn. \\
   \ \\

   {\small Qi Ye} \par
   School of Mathematical Sciences \par
   South China Normal University \par
   Guangzhou, 510631, Guangdong, PR China \\

   Pazhou Lab \par
   Guangzhou, 511442, Guangdong, PR China.\par
   Email-address: yeqi@m.scnu.edu.cn. \\

\end{document}